\documentclass[aip,reprint,footinbib,reqno]{amsart}

\usepackage{graphicx,graphics,caption, subcaption, enumerate,color}
\usepackage{float}
\usepackage{amsmath,amssymb,amsfonts,bm,empheq}
\usepackage{cases}

\usepackage{comment}
\usepackage{soul}

\usepackage{epstopdf}

\usepackage[a4paper]{geometry} 
\usepackage{hyperref}
 \hypersetup{
    colorlinks=true,       
    linkcolor=red,          
     citecolor=green,        
     urlcolor=cyan           
 }

\usepackage{setspace}

\usepackage{diagbox} 
\usepackage{multirow}

\usepackage{soul}
 \usepackage{cancel}
 \usepackage[normalem]{ulem}
\usepackage{todonotes}

\usepackage{lineno}

\newcommand{\set}[1]{\left\{#1\right\}}

\renewcommand{\d}{\ensuremath{\mathrm{d}}}

\newtheorem{theorem}{Theorem}

\newtheorem{lemma}[theorem]{Lemma}

\def\kpa2_2{\frac{\kappa^2}{2}}

\def\sigm2_2{\frac{\sigma^2}{2}}
\def\nab_dot{\nabla\cdot}

\def\Pj{\mathbf{P}}

\def\d{\rm{d}}

\def\1/N{\frac{1}{N}}

\def\H_1_u{H_u^{-1}}

\def\f12{\frac{1}{2}}

\newcommand{\mbx}{{\mathbf x}}
 
\begin{document}
 
\begin{center}
{\bf \Large Projection Method for Steady states of Cahn-Hilliard Equation with the Dynamic Boundary Condition}

\

Shuting Gu \footnote{ College of Big Data and Internet, Shenzhen Technology University, Shenzhen 518118, China. Email: gushuting@sztu.edu.cn } 
and Ming Xiao \footnote{
School of Science, Beijing University of Posts and Telecommunications, Beijing 100876, China. Email: xiaoming@bupt.edu.cn
}
and Rui Chen\footnote{ Corresponding author. School of Science,Key Laboratory of Mathematics and Information Networks (Ministry of Education), Beijing University of Posts and Telecommunications, Beijing 100876, China. Email: ruichen@bupt.edu.cn}

\end{center}
 \date{\today}

\section*{Abstract}

The Cahn-Hilliard equation is a fundamental model that describes phase separation processes of two-phase flows or binary mixtures. In recent years, the dynamic boundary conditions for the Cahn-Hilliard equation have been proposed and analyzed.
 Our first goal in this article is to present a projection method to locate the steady state of the CH equation with dynamic boundary conditions. The main feature of this method is that it only uses the variational derivative in the metric $L^2$ and not that  in the metric $H^{-1}$, thus significantly reducing the computational cost. In addition, the projected dynamics fulfill the important physical properties: mass conservation and energy dissipation. In the temporal construction of the numerical schemes, the convex splitting method is used to ensure a large time step size. Numerical experiments for the two-dimensional Ginzburg-Landau free energy, where the surface potential is the double well potential or the moving contact line potential, are conducted to demonstrate the effectiveness of this projection method.

\

{{\bf Keywords}: Chan-Hilliard equation, dynamic boundary condition, projection method, convex splitting scheme }
\

{{\bf  Mathematics Subject Classification (2010)}  Primary 65K05, Secondary  82B05	  }


\section{Introduction}
   
  
  
  The Cahn-Hilliard (CH) equation, first proposed in \cite{CH-EQ}, plays a crucial role in describing the phase separation process of the binary mixtures and the coarsening of a binary liquid in  the   porous Brinkman medium \cite{etde_21417454}. It is also widely used in many complicated moving interface problems in materials science and fluid dynamics through a phase field approach, for example, \cite{ShenYang, Fife, Bates, Brachet}.
  
   In recent years, many researchers have focused on studying the CH equation with different types of dynamic boundary conditions\cite{Chen2018a,Ma2017} to consider  possible interactions of the material with the solid wall at close range.
   And various corresponding models have been developed. For example, Goldstein et al. in 2011 \cite{Goldstein2011} proposed the CH equation with the dynamic boundary condition (called "GMS" model)  and proved the energy dissipation law and mass conservation; In 2019, Liu and Wu proposed a new dynamic boundary condition for the CH equation (the so-called "Liu-Wu" model), which assumed that there is no mass exchange between the bulk and the boundary; In 2021, Knopf et al. proposed another type of  dynamic boundary condition with a relaxation parameter $\gamma$.  It can be viewed as an interpolation between the GMS model and the Liu-Wu model. This model relaxes to the Liu-Wu model when $\gamma$ tends to infinity and relaxes to the GMS model when $\gamma$ tends to zero. In this article, we are also interested in the CH equation with the dynamic boundary condition and we start the discussion by recalling the Liu-Wu model \cite{Liu2019} as follows:   
\begin{subnumcases}{\label{Liu-Wu_model}}
   	\phi_t= -\gamma_1 \Delta \mu,  & \text { in } $\Omega \times(0, T)$, \label{equ:CH_1} \\
  \mu=  - \kappa^2 \Delta \phi+ f(\phi),  & \text { in }$ \Omega \times(0, T)$, \label{equ:CH_2} \\
  \partial_\mathbf{n} \mu = 0, &  \text { in } $\Omega \times(0, T)$, \label{equ:CH_3} \\
  \left.\phi\right|_{\Gamma}=\psi,  & \text { on } $ \Gamma \times(0, T)$, \label{equ:CH_4} \\
    \psi_t= -\gamma_2 \Delta_\Gamma\mu_{\Gamma},  & \text { on } $\Gamma \times(0, T)$, \label{equ:CH_5} \\
    \mu_{\Gamma} = - \kappa^2 \Delta_{\Gamma} \psi + g(\psi) + \partial_{\mathbf{n}} \phi,  & \text { on } $ \Gamma \times(0, T)$. \label{equ:CH_6}
\end{subnumcases}
Here, $\phi$ represents the order parameter, $T$ is a finite time, and $\Omega = \mathbb{R}^d$ is the bounded region with the boundary $\Gamma$, The vector $\mathbf{n}$ denotes the unit normal to $\Gamma$. The constants $\gamma_1>0, \gamma_2 > 0$ are relaxation parameters. The operator $\Delta_\Gamma$ represents the Laplace-Beltrami operator on $\Gamma$. The function $f(\phi) = F^\prime(\phi) = \phi(\phi^2-1)$ defines the derivative of the volume potential $F(\phi)$, while $g(\psi) = G^\prime(\psi)$ defines the derivative of the surface potential $G(\psi)$. The variables $\mu$ and $\mu_\Gamma$ represent the chemical potential in the bulk and on the boundary, respectively.

The detailed numerical analysis of weak and strong solutions for the Liu-Wu model has been established in \cite{Liu2019} and constructed using the gradient flow approach in \cite{experiment1}. Various numerical methods have been developed for approximating the CH equation with dynamic boundary conditions, including the stabilized factor method \cite{Bao2021a,Meng2023,Bao2021}, the finite element method \cite{Knopf2021, experiment1}, the SAV method \cite{Metzger2023} and other approaches \cite{Liu2019,Metzger2023,Bao2021,Bao2021a,Knopf2021,
Meng2023,Knopf2021a,Cherfils2010,Cherfils2014,Fukao2017,Israel2014}. These methods enable the construction of first-order or second-order linear and energy-stable numerical schemes. A common characteristic of these methods is that they solve the fourth-order CH equation directly. However,
this leads to high  computational costs due to the  fourth-order spatial derivative.
It is well known that the CH equation represents the gradient flow of  the Ginzburg-Landau free energy $E(\phi)$ in the $H^{-1}$ metric:
\begin{equation}\label{CH_eq}
	\frac{\partial\phi}{\partial t} = \Delta \delta_\phi E.
\end{equation}
If we consider the gradient flow in the $L^2$ metric, we obtain the Allen-Cahn (AC) equation, which involves a second-order spatial derivative:
 \begin{equation}\label{AC_eq}
 \frac{\partial\phi}{\partial t} = -\delta_\phi E,
 \end{equation}
 where $\delta_\phi E$ denotes the first-order variational derivative of the functional $E$ in the $L^2$ sense. 
   These two gradient flows exhibit distinct dynamics and properties. The primary difference is that the CH equation conserves mass, i.e.,  $\int_\Omega \phi \, \d \mbx$ remains constant,  while the AC equation does not. Based on this difference, 
a natural approach for handling the CH equation in the $H^{-1}$ metric is to incorporate a mass-conservation constraint into the AC equation in the $L^2$ metric.
 In \cite{Chen2024}, Chen et al.  implemented this idea by introducing penalty terms or Lagrangian multipliers into the AC equation to ensure the mass conservation. This approach effectively reduces the original higher-order equation to a lower-order one while preserving mass conservation. However, it also transforms the problem into a constrained one. In general, solving a constrained problem tends to incur higher  computational costs compared to  solving an unconstrained one.

 In this note, our goal is also to find the steady states of the fourth order CH equation with dynamic boundary conditions by reducing it to a second-order equation while preserving mass conservation. However, our primary motivation is distinctly different: mass conservation is imposed through a projection operator $\Pj$.
 This operator is an orthogonal projection onto a constrained linear subspace that inherently satisfies the mass conservation condition. For instance, when applied to the inner region of the domain, the projection operator $\Pj$ is defined as:
$\Pj u:= u - \frac{1}{|\Omega|}\int_\Omega u(\mathbf{x})\, \d \mathbf{x}$.
Thus, the projected AC equation becomes
 \begin{equation}\label{ProAC}
 	\frac{\partial u}{\partial t} = \Pj (-\delta_u F ).
 \end{equation}
 It is easy to verify that 
 both the  CH equation \eqref{CH_eq} and the projected AC equation \eqref{ProAC}
 preserve mass, although their  gradient-descent trajectories differ \cite{Tiejun_1D}. Morever,
  \eqref{CH_eq} and  \eqref{ProAC} share   the same stationary points (metastable states) and the saddle points, provided they initially have the same mass \cite{WEI1998459,Fife2000,ProjIMF, Fife, Bates,Tiejun2D, Tiejun_1D}. This can be easily demonstrated -- see the proof in \cite{ProjIMF}, where this relationship was used to locate the saddle points of the CH equation through a saddle point search method. 
 This idea was mentioned in \cite{LLinProj2010} as a counterpart to the CH equation for exploring different phases in diblock copolymers. 
 \cite{Tiejun2D}  further compared the stochastic models 
 arising from these two dynamics (\eqref{CH_eq} and \eqref{ProAC})  in the context of noise-induced transitions.  However,
 it is important to remind readers that the true dynamics of the CH equation \eqref{CH_eq} differ from those of the projected Allen-Cahn equation.
 
 The key novelty of this paper lies in introducing the projection operator to solve the CH equation with dynamic boundary conditions. This approach not only transforms the problem into a lower order differential equation, significantly reducing computational costs, but also avoids tackling a challenging constrained problem. 
 Instead, We only need to apply the orthogonal projection to enforce mass conservation.   Furthermore, in the time discretization process,  we adopt the convex splitting method to develop an unconditional energy-stable finite difference scheme, which allows for large time steps.
 In the numerical experiments, we demonstrate the effectiveness of this new method using various initial states and examine two cases  for the surface potential $G(\psi)$: the double-well potential and the moving contact line problem. 

  The paper is organized as follows. In Section \ref{Proj_AC}, we introduce the projected Allen-Cahn equation with dynamic boundary conditions and prove both energy dissipation and mass conservation. Section \ref{Scheme} presents the time-discrete numerical scheme, constructed using the convex splitting method. In section \ref{Num_ex}, we show four  numerical experiments with varying initial conditions and different surface potential $G(\phi)$. Finally,  the conclusions are discussed in Section \ref{Conclusion}.


  \section{Projected AC eq. with dynamic boundary condition}\label{Proj_AC}

   
We  consider the total free energy of the system, which includes both  bulk energy and surface energy:
  \begin{equation}\label{E_total}
  	E^{\text{total}}(\phi, \psi) = E^{\text{bulk}}(\phi) + E^{\text{surf}}(\psi),
  \end{equation}
  where
  \begin{align}
  	E^{\text{bulk}}(\phi) & = \int_\Omega \frac{\kappa^2}{2}|\nabla \phi|^2 + F(\phi \, {\d} {\bf{x}},\label{E_bulk}\\
  	E^{\text{surf}}(\psi)  & = \int_{\Gamma}  \frac{\kappa^2}{2}|\nabla_\Gamma \psi|^2 + G(\psi) \, {\d} S,\label{E_surf}
  \end{align}
  where $\kappa$ represents the width of the interface.  The bulk potential is given by $F(\phi) = \frac{1}{4}(\phi^2-1)^2$, while $G(\phi)$ is the surface potential, which can be either a double-well potential $G(\psi) = \frac{1}{4}(\psi^2-1)^2$ or a moving contact line potential $G(\phi) =  -\frac{\tilde\gamma}{2} \cos(\theta_s) \sin(\frac{\pi}{2}\phi)$. 
  The operator $\nabla_\Gamma$ denotes the tangential gradient operator on the boundary $\Gamma$.  

  The gradient flow of $E^{\text{total}}$ in the $H^{-1}$ metric corresponds exactly  to the Liu-Wu model \eqref{Liu-Wu_model} with dynamic boundary conditions. Our interest lies in finding the steady states of this equation. In this section,  we address the problem by introducing a projected Allen-Cahn equation with a corresponding projected dynamic boundary condition. 
 
  Since mass is conserved in the gradient flow under $H^{-1}$ metric, we define the projection operators $\Pj_1$ and $\Pj_2$ as follows:
  \begin{equation}\label{Pj}
  	\Pj_1 u:= u - \frac{1}{|\Omega|}\int_\Omega u({\bf{x}})\, {\d} {\bf{x}}, \quad \Pj_2 u := u - \frac{1}{|\Gamma|}\int_\Gamma u({\bf{x}})\, {\d} S,
  \end{equation}
  onto the linear subspaces
  $$H_1=\set{u \in L^2(\Omega) : \int_\Omega u({\bf{x}}) \d {\bf{x}} =0} \text{and} ~~  H_2=\set{u \in L^2(\Gamma) : \int_\Gamma u({\bf{x}}) {\d} S =0},$$
  respectively.  
  Let $\langle \cdot, \cdot \rangle$ denote the $L^2$ inner product.
  It is straightforward to show that $\Pj_i$ satisfies the following properties, for $i=1,2,$
  \begin{enumerate}
  	\item $\Pj_i^2 = \Pj_i$;
  	\item $\Pj_i u \in H_i, \forall u \in L^2$;
  	\item $\Pj_i v = v, \forall v \in H_i$;
  	\item  $ \langle \Pj_i u, u \rangle \geq 0,  \forall u \in L^2 $. 
  \end{enumerate}
The first three properties are straightforward and can be found in \cite{ProjIMF} for a detailed proof. As for the last property, we only give the simple proof for $i=1$. In fact, $ \forall u \in L^2(\Omega), $ we have
\begin{align*}
    \langle \Pj_1 u, u \rangle = & \int_\Omega \Pj_1 u \cdot u ~ \d \mathbf{x} \\
    = &  \int_\Omega \left( u - \frac{1}{|\Omega|}\int_\Omega u({\mathbf{x}})\, \d {\mathbf{x}} \right ) u({\bf{x}}) \d {\bf{x}}\\
    = &  \int_\Omega u^2 {\d} {\bf{x}} - \frac{1}{|\Omega|} \left(\int_\Omega u({\bf{x}}) \d{\bf{x}} \right)^2 \\
    = &  \frac{1}{|\Omega|} \left( \int_\Omega u^2 {\d} {\bf{x}} \cdot \int_\Omega 1^2 {\d}{\bf{x}} - \big(\int_\Omega u({\bf{x}}) \cdot 1 ~ \d {\bf{x}} \big)^2  \right) \geq 0,
\end{align*}
by the Cauchy-Schwarz inequality.
  
  Similarly to the equivalence of the steady states between the CH equation \eqref{CH_eq} and the projected AC equation \eqref{ProAC}, all the terms computed in the $H^{-1}$ metric in the Liu-Wu model \eqref{Liu-Wu_model} can be transformed into corresponding terms in $L^2$ metric by projecting onto the confined subspace $H_0$. Consequently, the projected Allen-Cahn(Proj-AC) equation, along with the projected Allen-Cahn type dynamic boundary condition, can be derived as follows:
  \begin{subnumcases}{\label{Proj_Liu-Wu}}
	 \phi_t= -\gamma_1 \Pj_1 \mu, & \text { in } $ \Omega \times(0, T),$ \label{equ:ProjAC_bulk} \\
 	\mu=  - \kappa^2 \Delta \phi+ f(\phi), & \text { in } $\Omega \times(0, T), $ \\
	\partial_\mathbf{n} \mu = 0,  & \text { in } $ \Omega \times(0, T), $ \\
	\left.\phi\right|_{\Gamma}=\psi, & \text { on } $ \Gamma \times(0, T), $ \label{phi_boundary}\\
	\psi_t= -\gamma_2 \Pj_2 \mu_{\Gamma}, & \text { on } $ \Gamma \times(0, T),$ \label{equ:ProjAC_surf} \\
	\mu_{\Gamma}=- \kappa^2 \Delta_{\Gamma} \psi+g(\psi) + \kappa^2 \partial_{\mathbf{n}} \phi, & \text { on } $ \Gamma \times(0, T), $
	\end{subnumcases}
 where $f(\phi) = \phi^3-\phi$ and  $g(\psi) = G^\prime(\phi)$, 
   it is clear that the projected system \eqref{Proj_Liu-Wu} reduces the spatial derivative order by two compared to the original Liu-Wu model \eqref{Liu-Wu_model}, which operates directly as a gradient flow in the $H^{-1}$ metric. Moreover, it can be easily demonstrated  that the continuous model \eqref{Proj_Liu-Wu} preserves mass and ensures energy stability,  as formalized in the following two theorems.
 
  \begin{theorem}
 (Mass conservation)  	
  The solution of the projected AC equation \eqref{Proj_Liu-Wu} with projected AC type dynamic boundary conditions preserves  mass conservation in both the bulk and the surface as time evolves. Specifically, we have
  \begin{align*}
      \frac{\d}{{\d} t} \int_\Omega \phi \, {\d} \mbx = 0, \quad  \frac{\d}{{\d} t} \int_\Gamma \psi \, {\d} {S} = 0.
  \end{align*}
  
  \end{theorem}
  
  \begin{proof}
  	Integrating \eqref{equ:ProjAC_bulk} and \eqref{equ:ProjAC_surf} on both sides, we obtain
  	\begin{align*}
  	    \frac{\d}{{\d} t} \int_\Omega \phi \, {\d} \mbx = \gamma \int_\Omega  \Pj_1 (\Delta \phi - f(\phi)) \, \d \mbx = 0, 
  	\end{align*}
  	and 
  	\begin{align*}
  	    \frac{\d}{{\d} t} \int_\Gamma \psi \, {\d} S = M \int_\Gamma  \Pj_2 (\Delta_\Gamma \psi - g(\phi) - \partial_\mathbf{n} \phi ) \, {\d} S = 0, 
  	\end{align*}
	utilizing the second property stated above.
  	
  \end{proof}

  \begin{theorem}
  	(Energy decay) 
   If $ \phi$ is the solution of the projected Allen-Cahn equation \eqref{Proj_Liu-Wu}, then the following inequality holds for the total energy functional $E^{\text{total}}$:
   \begin{equation}
       \frac{ {\d} E^{\text{total}}}{{\d} t} \leq 0.
   \end{equation}
  \end{theorem}

  \begin{proof}
      On one hand, taking the time derivative for the total energy $E^{\text{total}}$ \eqref{E_total}, we have
      \begin{align}
          \frac{{\d} E^{\text{total}} }{{\d}t} = & \int_\Omega \kappa^2 \nabla \phi  \cdot \nabla \frac{\partial \phi}{\partial t}  + f(\phi) \frac{\partial \phi}{\partial t}  \, {\d} {\bf{x}} +  \int_{\Gamma}  \kappa^2 \nabla_\Gamma \psi \cdot \nabla_\Gamma (\frac{\partial \psi}{\partial t}) + g(\psi) \frac{\partial \psi}{\partial t}  \, {\d} S \nonumber\\
          = &  \int_\Omega - \kappa^2 \Delta \phi  \cdot \frac{\partial \phi}{\partial t} \, {\d} {\bf{x}} +  \int_\Gamma \kappa^2 \partial_{\mathbf{n}} \phi \cdot \frac{\partial \phi}{\partial t} \, {\d} S + \int_\Omega  f(\phi) \frac{\partial \phi}{\partial t}  \, {\d} {\bf{x}} \nonumber\\
          & + \int_{\Gamma}  -\kappa^2 \Delta_\Gamma \psi \cdot \frac{\partial \psi}{\partial t} + g(\psi) \frac{\partial \psi}{\partial t}  \, {\d} S \nonumber\\
          = &  \int_\Omega \Big( - \kappa^2 \Delta \phi 
          +  f(\phi) \Big) \cdot \frac{\partial \phi}{\partial t} \, {\d} \bf{x} \nonumber\\
         &  + \int_\Gamma  \Big( -\kappa^2 \Delta_\Gamma \psi + g(\psi) + \kappa^2 \partial_{\mathbf{n}} \phi \Big) \cdot \frac{\partial \psi}{\partial t} \, {\d} S, \label{Egy_decay}
      \end{align}
      where, the second equality employs the Gauss-Green formula, and the last equality is derived using \eqref{phi_boundary}.

      On the other hand, taking the $L^2$ inner product of \eqref{equ:ProjAC_bulk} and \eqref{equ:ProjAC_surf} with $\mu$ and $\mu_\Gamma$ in the bulk and the surface, respectively, yields
      \begin{align}\label{Egy_decay1}
        \int_\Omega \frac{\partial \phi}{\partial t} \cdot \Big( - \kappa^2 \Delta \phi 
          +  f(\phi) \Big)  \, {\d} {\bf{x}} = \langle -\gamma_1 \Pj_1 \mu, \mu \rangle_\Omega \leq 0,
      \end{align}
      and
      \begin{align}\label{Egy_decay2}
         \int_\Gamma \frac{\partial \psi}{\partial t} \cdot \Big( - \kappa^2 \Delta_\Gamma \psi 
          +  g(\psi) + \kappa^2 \partial_{\bf{n}} \phi \Big)  \, {\d} S = \langle -\gamma_2 \Pj_2 \mu_\Gamma, \mu_\Gamma \rangle_\Gamma  \leq 0,
      \end{align}
      where $\langle \cdot, \cdot \rangle_\Omega$ and $\langle \cdot, \cdot \rangle_\Gamma$ denote the $L^2$ inner product in the bulk and the surface, respectively.
      By combing \eqref{Egy_decay}, \eqref{Egy_decay1} and \eqref{Egy_decay2}, we obtain the energy dissipation law:
      \[ \frac{{\d} E^{\text{total}}}{ {\d} t} \leq 0 .\]
      
  \end{proof}

  \section{Time-Discrete scheme using the convex splitting method}\label{Scheme}
  
  
  In this part, we propose an unconditional energy stable finite difference scheme for the projected AC type Liu-Wu model  \eqref{Proj_Liu-Wu} using the convex splitting method. We will also discuss the mass conservation property and the energy dissipation law for the time-discrete scheme. For the sake of convenience in narration, we consider the double-well potential $G(\psi) = \frac{1}{4}(\psi^2-1)^2$ as the surface potential in this section.

    \begin{lemma}\label{convex_forms}
        Suppose $S_1$ and $S_2$ are two positive parameters, satisfying $S_1 \geq \frac{1}{2}\max\limits_{\phi\in\mathbb{R}} F^{\prime\prime}(\phi),$ $ S_2 \geq \frac{1}{2}\max\limits_{\psi\in\mathbb{R}} G^{\prime\prime}(\psi)$, where $F(\phi) $ and $G(\psi)$ are the double well potentials defined above. We define the following energies:
        \begin{align}
            E_c^{\text{bulk}} & = \int_\Omega \frac{\kappa^2}{2}|\nabla \phi|^2 + \frac{1}{2} S_1 \phi^2 + \frac{1}{4} \, \d \bf{x},\\
            E_e^{\text{bulk}} & = \int_\Omega \frac{1}{2} (S_1+1) \phi^2 - \frac{1}{4}\phi^4 \, \d \bf{x},\\
           E_c^{\text{surf}} & = \int_\Gamma \frac{\kappa^2}{2}|\nabla_\Gamma \psi|^2 + \frac{1}{2} S_2  \psi^2 + \frac{1}{4} \, {\d} S,\\
            E_e^{\text{surf}} & = \int_\Gamma \frac{1}{2} (S_2+1) \psi^2 - \frac{1}{4}\psi^4 \, {\d} S.
        \end{align}
        Then $ E_c^{\text{bulk}} $ and $ E_e^{\text{bulk}}$ are both convex w.r.t. $\phi$, and $ E_c^{\text{surf}}$ and $E_e^{\text{surf}} $ are both convex w.r.t. $\psi$. 
    \end{lemma}
    \begin{proof}
        It is straightforward to calculate the first order variational derivatives of $E_c^{bulk}$ and $E_e^{bulk}$ w.r.t. $\phi$: 
        \begin{align*}
            \delta_\phi E_c^\text{bulk} = -\kappa^2 \Delta \phi + S_1 \phi, \quad \quad
            \delta_\phi E_e^\text{bulk} = (S_1+1)\phi - \phi^3.
        \end{align*}
    The second order variational derivatives are given by:
        \begin{align*}
            \delta^2_\phi E_c^\text{bulk} = -\kappa^2 \Delta + S_1, \quad \quad
            \delta^2_\phi E_e^\text{bulk} = S_1 - (3\phi^2-1).
        \end{align*}
        Both $\delta^2_\phi E_c^\text{bulk}$  and $\delta^2_\phi E_e^\text{bulk}$ are positive definite under the given conditions, confirming  the convexity of $E_c^\text{bulk}$ and $E_e^\text{bulk}$.
        
        Similarly, we can verify the positive definite nature of the second order variational derivatives of $E_c^\text{surf}$ and $E_e^\text{surf}$ as
        \begin{align*}
            \delta^2_\psi E_c^\text{surf} = -\kappa^2 \Delta + S_2, \quad \text{and} \quad
            \delta^2_\psi E_e^\text{surf} = S_2 - (3\psi^2-1),
        \end{align*}
        which also indicates the convexity of $E_c^\text{surf}$ and $E_e^\text{surf}$.
        
    \end{proof}

Based on the convex forms in Lemma \ref{convex_forms}, the total free energy in the system can be split into two convex parts as follows:
\begin{equation}\label{E_split}
    E^\text{total} = E_c - E_e,
\end{equation}
where
\begin{align*}
    E_c = E_c^\text{bulk} + E_c^\text{surf}, \quad \text{and} \quad 
    E_e = E_e^\text{bulk} + E_e^\text{surf}.
\end{align*}
By treating  the $E_c$ part implicitly and the $E_e$ part explicitly,  we can derive the time-discrete numerical scheme for the system \eqref{Proj_Liu-Wu}:
    \begin{subnumcases}{\label{convex_scheme}}
        \frac{\phi^{n+1}-\phi^n}{\Delta t}  =  \gamma_1 \Pj_1 \left( \kappa^2 \Delta\phi - S_1 \phi \right)^{n+1} \!+\! \gamma_1 \Pj_1 \left( -\phi^3 + (S_1+1) \phi \right)^n, \!\! & \text { in }  $\Omega$, \\
        \frac{\psi^{n+1}-\psi^n}{\Delta t}  =  \gamma_2 \Pj_2 \left( \kappa^2 \Delta_\Gamma\psi - S_2 \psi + \kappa^2 \partial_{\mathbf{n}} \phi \right)^{n+1} \!+\! \gamma_2 \Pj_2 \left( -\psi^3 + (S_2+1) \psi \right)^n,  \!\! & \text { on }  $\Gamma$ .    
    \end{subnumcases}
 where the parameters $S_1$ and $S_2$ satisfy the conditions outlined in Lemma \ref{convex_forms}.

 \begin{theorem}
      Suppose the total free energy $E^{\text{total}}$ is split into two parts, $E^{\text{total}} = E_c - E_e$, as in \eqref{E_split}. Then the time-discrete scheme \eqref{convex_scheme} is unconditionally energy stable. This means that for any time step size $\Delta t >0$, the following inequality holds:
      \begin{equation}
          E^{\text{total}}(\phi^{k+1},\psi^{k+1}) \leq           E^{\text{total}}(\phi^{k},\psi^k).
      \end{equation}
  \end{theorem}
  
  \begin{proof}
      Multiplying both sides of \eqref{convex_scheme} by  $\Delta t$ and introducing the notation $\nu_1$ and $\nu_2$ as follows:
      \begin{align*}
          \nu_1 &\triangleq \left[ \kappa^2 \Delta\phi - S_1 \phi \right]^{n+1} \!+\! \left[ -\phi^3 + (S_1+1) \phi \right]^n,\\
          \nu_2  &\triangleq \left[ \kappa^2 \Delta_\Gamma\psi - S_2 \psi + \kappa^2 \partial_{\mathbf{n}} \phi \right]^{n+1} \!+\! \left[ -\psi^3 + (S_2+1) \psi \right]^n.
      \end{align*}
    We can rearrange  \eqref{convex_scheme} as:
       \begin{subnumcases}{\label{convex_scheme_1}}
        \phi^{n+1}-\phi^n  =  \Delta t \gamma_1 \Pj_1 \nu_1,  & \text { in }  $\Omega$, \label{phi_scheme}\\
        \psi^{n+1}-\psi^n  =  \Delta t\gamma_2 \Pj_2 \nu_2,  & \text { on }  $\Gamma$. \label{psi_scheme}   
    \end{subnumcases}
    This formulation introduces the updated time-discrete system in terms of the variables $\nu_1$ and $\nu_2$, simplifying the numerical scheme.

    On the one hand, taking the $L^2$ inner product of \eqref{phi_scheme} and \eqref{psi_scheme} with $-\nu_1$ and $-\nu_2$, respectively, we have
    
    For the bulk term: 
    \begin{align*}
       0 \geq &~  \Delta t \gamma_1 \langle \Pj_1 \nu_1, 
        -\nu_1 \rangle_\Omega  \\
        = & ~ \langle \phi^{n+1} - \phi^n, -\nu_1 \rangle_\Omega \\
         = &~ \langle \phi^{n+1} - \phi^n,  -\kappa^2 \Delta\phi^{n+1} + S_1 (\phi^{n+1}- \phi^n) + (\phi^n)^3 - \phi^n \rangle_\Omega \\
         = & ~ \int_\Omega \kappa^2 \nabla \phi^{n+1} \cdot (\nabla\phi^{n+1} - \nabla\phi^n) \, {\d} {\bf{x}} + \int_\Gamma \kappa^2 \partial_{\mathbf{n}} \phi^{n+1} (\phi^{n+1} - \phi^n) \, {\d} S\\
         & + \int_\Omega f(\phi^n)(\phi^{n+1} - \phi^n) \, {\d} {\bf{x}} + S_1 \|\phi^{n+1} - \phi^n\|_\Omega^2,
    \end{align*}
    where $f(\phi^n) = (\phi^n)^3 - \phi^n. $ 
   The first term can be rewritten as
    \begin{align*}
         \frac{\kappa^2}{2} \int_\Omega |\nabla\phi^{n+1}|^2 - |\nabla\phi^n|^2 + |\nabla(\phi^{n+1} - \phi^n)|^2 \, \d {\bf{x}}.
    \end{align*}
    Thus we obtain
    \begin{align}\label{phi_ineq}
        &  \Delta t \gamma_1 \langle \Pj_1 \nu_1, 
        -\nu_1 \rangle_\Omega = \int_\Omega \frac{\kappa^2}{2} |\nabla\phi^{n+1}|^2 - \frac{\kappa^2}{2} |\nabla\phi^n|^2  + f(\phi^n)(\phi^{n+1} - \phi^n) \, \d {\bf{x}} \nonumber\\
        & +S_1 \|\phi^{n+1} - \phi^n\|_\Omega^2 + \frac{\kappa^2}{2} \|\nabla(\phi^{n+1} - \phi^n) \|^2_\Omega + \int_\Gamma \kappa^2 \partial_{\mathbf{n}} \phi^{n+1} (\phi^{n+1} - \phi^n) \, {\d} S \leq 0.
    \end{align}
    Similarly,  for the surface term:
     \begin{align}\label{psi_ineq}
        0 \geq &~  \Delta t \gamma_2 \langle \Pj_2 \nu_2, 
        -\nu_2 \rangle_\Gamma  \nonumber\\
         = &~ \langle \psi^{n+1} - \psi^n,  -\kappa^2 \Delta_\Gamma \psi^{n+1} + S_2 (\psi^{n+1}- \psi^n) + (\psi^n)^3 - \psi^n - \kappa^2 \partial_{\mathbf{n}}\phi^{n+1} \rangle_\Gamma \nonumber\\
         = & ~ \int_\Gamma \kappa^2 \nabla_\Gamma \psi^{n+1} \cdot (\nabla_\Gamma \psi^{n+1} - \nabla_\Gamma \psi^n) \, {\d} S + \int_\Gamma g(\psi^n)(\psi^{n+1} - \psi^n) \, {\d} S \nonumber\\
         & - \int_\Gamma \kappa^2 \partial_{\mathbf{n}} \phi^{n+1} (\psi^{n+1} - \psi^n) \, {\d} S
          + S_2 \|\psi^{n+1} - \psi^n\|_\Gamma^2 \nonumber\\
        = & \int_\Gamma \frac{\kappa^2}{2} |\nabla_\Gamma\psi^{n+1}|^2 - \frac{\kappa^2}{2} |\nabla_\Gamma\psi^n|^2 + g(\psi^n)(\psi^{n+1} - \psi^n) \, {\d} S + S_2 \|\psi^{n+1} - \psi^n\|_\Gamma^2 \nonumber\\
        &  + \frac{\kappa^2}{2} \|\nabla_\Gamma(\phi^{n+1} - \phi^n)\|^2_\Gamma - \int_\Gamma \kappa^2 \partial_{\mathbf{n}} \phi^{n+1} (\psi^{n+1} - \psi^n) \, {\d} S \leq 0.
    \end{align}
    On the other hand, from \eqref{E_bulk} and \eqref{E_surf}, we can express the changes in energy as follows:

    For the bulk energy:
    \begin{align}\label{Ebulk}
        & E^\text{bulk}(\phi^{n+1}) - E^\text{bulk}(\phi^{n}) \nonumber\\
         = & \int_\Omega \frac{\kappa^2}{2} |\nabla\phi^{n+1}|^2 - \frac{\kappa^2}{2} |\nabla\phi^n|^2  + F(\phi^{n+1}) - F(\phi^n) \, {\d} {\bf{x}} \nonumber\\
         = & \int_\Omega \frac{\kappa^2}{2} |\nabla\phi^{n+1}|^2 - \frac{\kappa^2}{2} |\nabla\phi^n|^2  + f(\phi^n)(\phi^{n+1} - \phi^n) + \frac{F^{\prime\prime}(\xi)}{2}(\phi^{n+1} - \phi^n)^2 \, {\d} {\bf{x}} \nonumber\\
          = & \int_\Omega \frac{\kappa^2}{2} |\nabla\phi^{n+1}|^2 - \frac{\kappa^2}{2} |\nabla\phi^n|^2  + f(\phi^n)(\phi^{n+1} - \phi^n) \, {\d} {\bf{x}} + \frac{F^{\prime\prime}(\xi)}{2}\|\phi^{n+1} - \phi^n\|^2_\Omega,
    \end{align}
    where $\xi$ is between $\phi^n$ and $\phi^{n+1}$.

    For the surface energy:
    \begin{align}\label{Esurf}
        & E^\text{surf}(\psi^{n+1}) - E^\text{surf}(\psi^{n}) \nonumber\\
         = & \int_\Gamma \frac{\kappa^2}{2} |\nabla_\Gamma\psi^{n+1}|^2 - \frac{\kappa^2}{2} |\nabla_\Gamma\psi^n|^2  + G(\psi^{n+1}) - G(\psi^n) \, {\d} {\bf{x}} \nonumber\\
         = & \int_\Gamma \frac{\kappa^2}{2} |\nabla_\Gamma\psi^{n+1}|^2 - \frac{\kappa^2}{2} |\nabla_\Gamma\psi^n|^2  + g(\psi^n)(\psi^{n+1} - \psi^n) + \frac{G^{\prime\prime}(\eta)}{2}(\psi^{n+1} - \psi^n)^2 \, {\d} S \nonumber \\
          = & \int_\Gamma \frac{\kappa^2}{2} |\nabla_\Gamma\psi^{n+1}|^2 - \frac{\kappa^2}{2} |\nabla_\Gamma\psi^n|^2  + g(\psi^n)(\psi^{n+1} - \psi^n)  \, {\d} S + \frac{G^{\prime\prime}(\eta)}{2} \|\psi^{n+1} - \psi^n\|^2_\Gamma,
    \end{align}
    where  $\eta$ is between $\psi^n$ and $\psi^{n+1}$. The second order Tylor expansion of $F(\phi^{n+1})$ and $G(\psi^{n+1})$ at $\phi^n$ and $\psi^n$ are applied for
    the last equalities in \eqref{Ebulk} and \eqref{Esurf}. 
    
    Adding \eqref{phi_ineq} and \eqref{psi_ineq} along with \eqref{Ebulk} and \eqref{Esurf}, and 
    considering $\phi = \psi, \text{on} ~ \Gamma$, we obtain:
    \begin{align*}
    & E^{\text{total}}(\phi^{n+1},\psi^{n+1}) -         E^{\text{total}}(\phi^{n},\psi^n) \\
    = ~ & E^\text{bulk}(\phi^{n+1}) - E^\text{bulk}(\phi^n) + E^\text{surf}(\psi^{n+1}) - E^\text{surf}(\psi^n) \\
    = ~ &  \Delta t \gamma_1 \langle \Pj_1 \nu_1, 
        -\nu_1 \rangle_\Omega +  \Delta t \gamma_2 \langle \Pj_2 \nu_2, 
        -\nu_2 \rangle_\Gamma - \frac{\kappa^2}{2} \|\nabla(\phi^{n+1} - \phi^n) \|^2_\Omega -  \frac{\kappa^2}{2} \|\nabla_\Gamma(\phi^{n+1} - \phi^n)\|^2_\Gamma \\
    & - (S_1 - \frac{F^{\prime\prime}(\xi)}{2}) \|\phi^{n+1} - \phi^n \|^2_\Omega - (S_2 - \frac{G^{\prime\prime}(\eta)}{2}) \|\psi^{n+1} - \psi^n\|^2_\Gamma),
    \end{align*}
    
    Considering the conditions 
    $ S_1 \geq \frac{1}{2}\max\limits_{\phi\in\mathbb{R}} F^{\prime\prime}(\phi), S_2 \geq \frac{1}{2}\max\limits_{\psi\in\mathbb{R}} G^{\prime\prime}(\psi)$ from Lemma \ref{convex_forms}, we conclude:
    $$ E^{\text{total}}(\phi^{k+1},\psi^{k+1}) -         E^{\text{total}}(\phi^{k},\psi^k) \leq 0, $$
    which completes the proof.
      
  \end{proof}

  \section{Numerical experiments}\label{Num_ex}

In this section, we demonstrate the projection method by locating the steady state of the original Liu-Wu model \eqref{Liu-Wu_model} through the projected Liu-Wu model \eqref{Proj_Liu-Wu}. We test three different initial states for the double-well  surface potential $G(\phi)$, and two initial cases for the moving contact line surface potential. 
The results are presented from two perspectives: energy decay and mass conservation, both in the bulk and the surface.


In the numerical implementation of the scheme in \eqref{convex_scheme}, we use the following parameters: $\gamma_1=100, \gamma_2 = 100, S_1 =100, $ and $ S_2 = 100.$ The time step size is $\Delta t=0.001$, with mesh points $N_x=N_y=200$, grid spacing $h=1/N_x$, and $\kappa = 2h$.

\subsection{Example 5.1}

In the first example, the initial state is defined as:
\begin{equation}\label{ex1_phi0}
    \phi_0({\bf{x}}) = \begin{cases}
        0,  & \text{in $\Omega=[0,1]\times [0,1]$,}\\
        1, & \text{on $\Gamma$.}
    \end{cases}
\end{equation}

Figure \ref{ex1_phi_t} shows the contour plots of $\phi$ at various times $T=0.1, 0.5, 2 $ and $7$, respectively, under the initial condition \eqref{ex1_phi0}. As time progresses, driven by the projected force, the 0-phase region in the domain gradually shrinks, while the 1-phase region on the boundary initially expands, in accordance with mass conservation. Ultimately, the $\phi$ field evolves into a circular -1-phase region surrounded by a 1-phase region, as seen in Figure \ref{ex1_phi_t} (D)). It is important to note that the blue region in the initial state represents the 0-phase, while in the steady state, it represents the -1-phase.
In Figure \ref{num_ex1}, we validate the mass conservation and energy decay properties by plotting the time evolution of mass and energy. In Figure \ref{num_ex1}(A), the red and blue lines  represent the mass in the bulk and on the boundary, respectively. Figure \ref{num_ex1}(B) illustrates the decay of the total energy over time.

 \begin{figure}[H]
    \centering
    \includegraphics[width=0.23\linewidth]{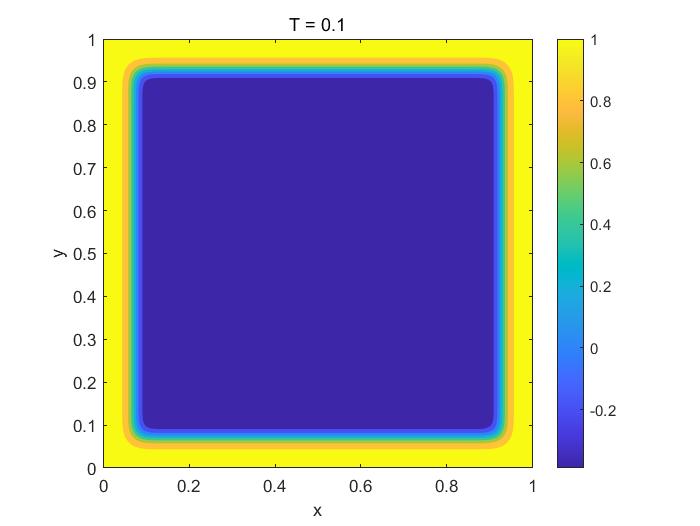}
    \includegraphics[width=0.23\linewidth]{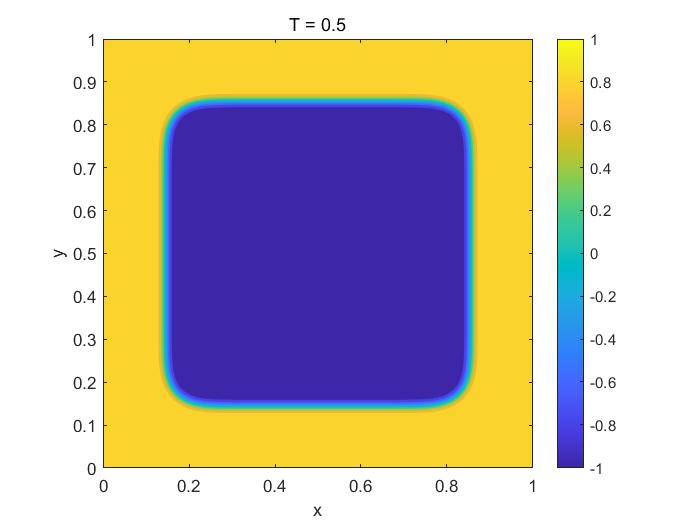}
    \includegraphics[width=0.23\linewidth]{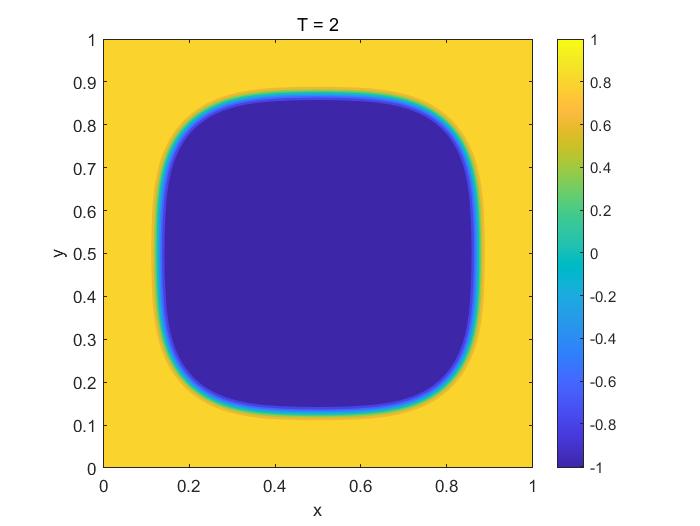}
    \includegraphics[width=0.23\linewidth]{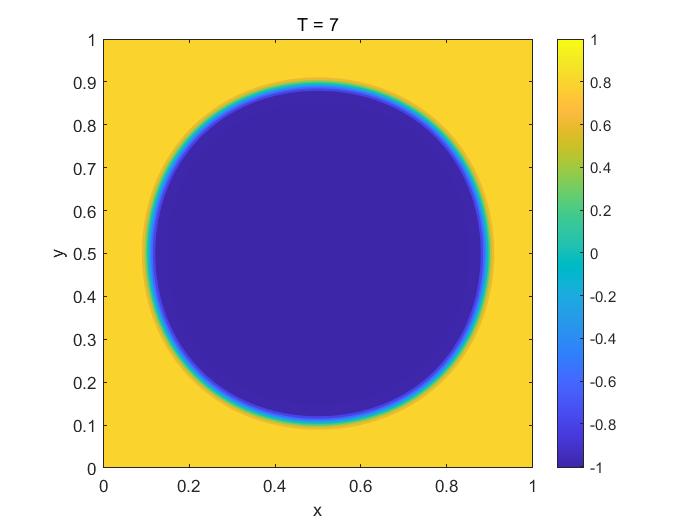}\\
    \hspace{-0.5cm} (A) $t=0.1$ \hspace{1.3cm} (B) $t=0.5$ \hspace{1.4cm} (C) $t=2$ \hspace{1.6cm} (D) $t=7$
    \caption{The contour plot of $\phi$ for the first example at times t=0.1, 0.5, 2, 7.}\label{ex1_phi_t}
\end{figure}

\begin{figure}[htbp]
\centering
\subfloat[mass evolution]{\label{fig:mass_ex1}\includegraphics[scale=0.11]{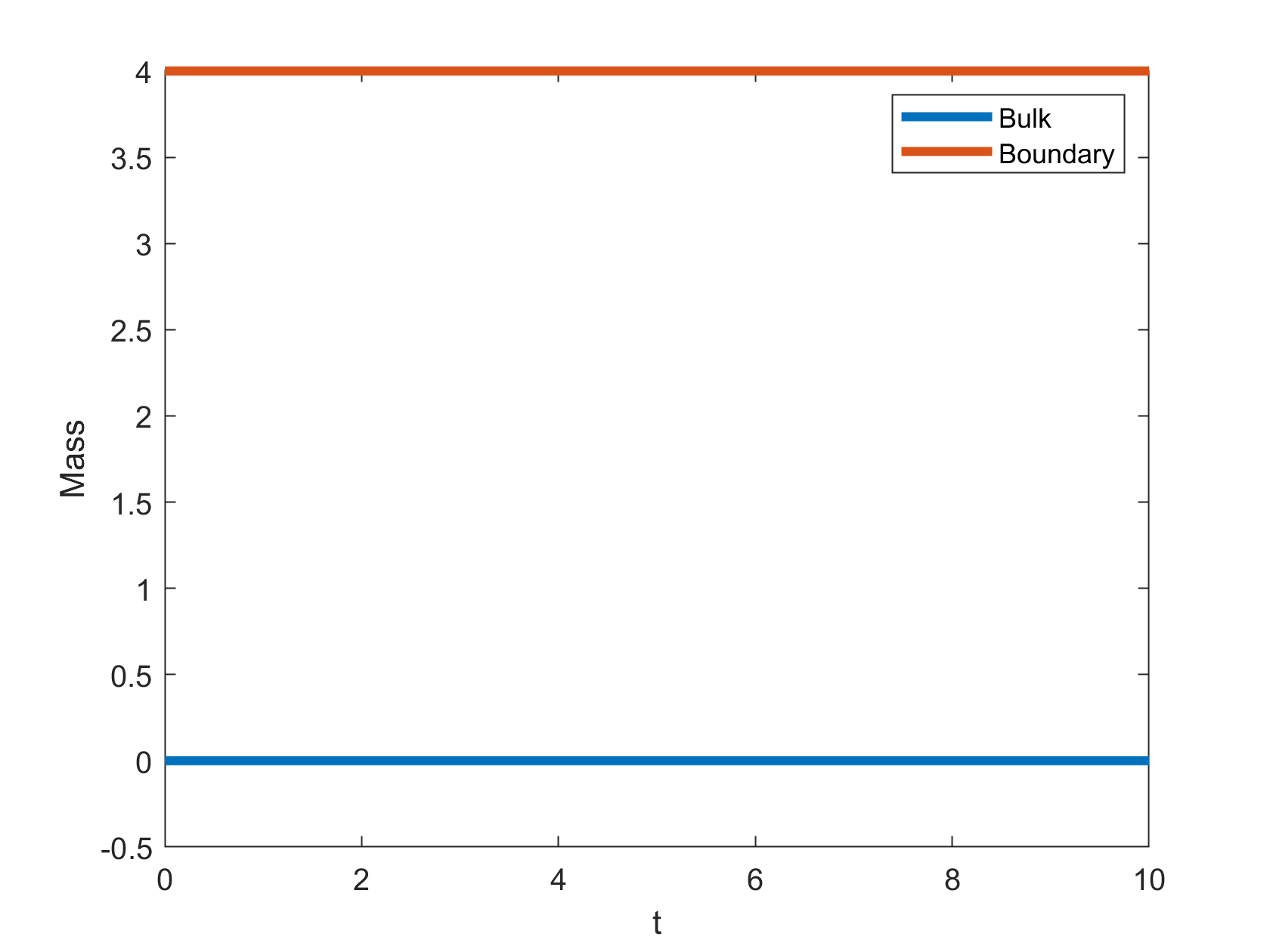}}
\subfloat[energy decay]{\label{fig:energy_ex1}\includegraphics[scale=0.11]{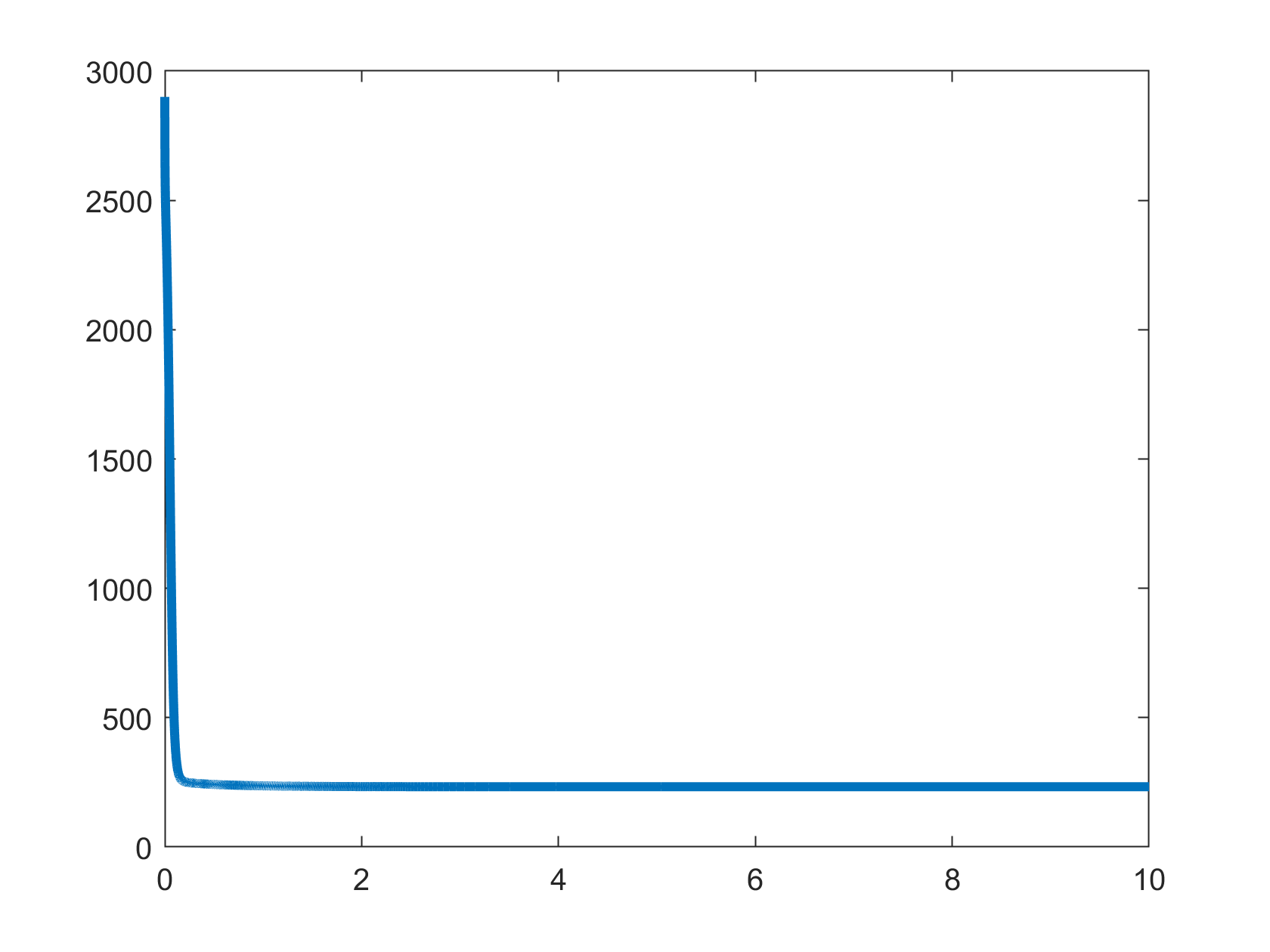}}
\caption{(A): Mass evolution over time in the bulk (blue line) and on the boundary (red line) for the first numerical example.
 (B): Energy decay over time. }
 \label{num_ex1}
\end{figure}

\subsection{Example 5.2}
In the second example, we choose the initial state as follows:
\begin{equation}
    \phi_0({\bf{x}}) = \begin{cases}
        1,  & {\bf{x}} \in \Omega_1,\\
        -1, & {\bf{x}} \in \Omega \backslash \Omega_1.
    \end{cases}
\end{equation}
where $\Omega = [0,1]\times [0,1], \Omega_1 = [0.25,0.75]\times [0,0.5].$

Figure \ref{ex2_phi_t} displays the evolution of $\phi$ at different times $T=0.1, 1, 3 $ and $7$, respectively. A square-shaped 1-phase region appears initially along one boundary of the domain, and it evolves into a circular shape at the boundary in the steady state finally. 
Figure \ref{num_ex2}(A) demonstrates the conservation of mass in both the bulk and on the boundary, while Figure \ref{num_ex2} (B) depicts the decay of the total energy over time.

 \begin{figure}[H]
    \centering
    \includegraphics[width=0.23\linewidth]{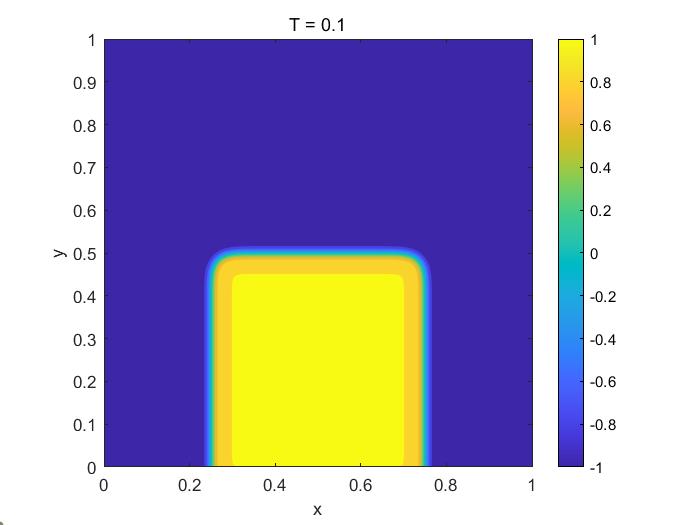}
    \includegraphics[width=0.23\linewidth]{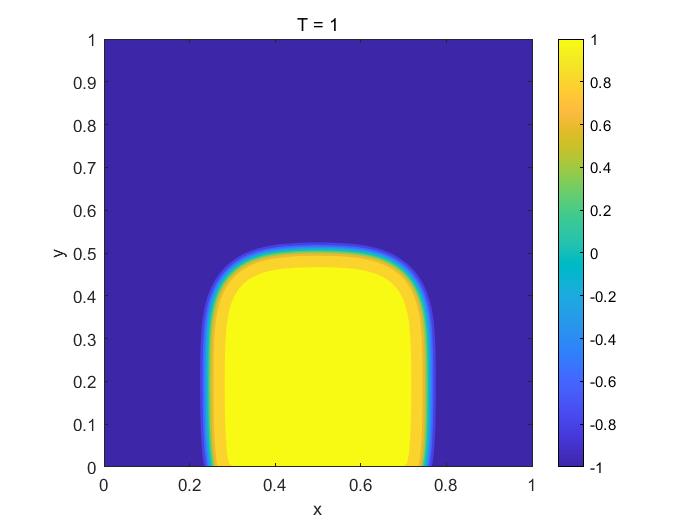}
    \includegraphics[width=0.23\linewidth]{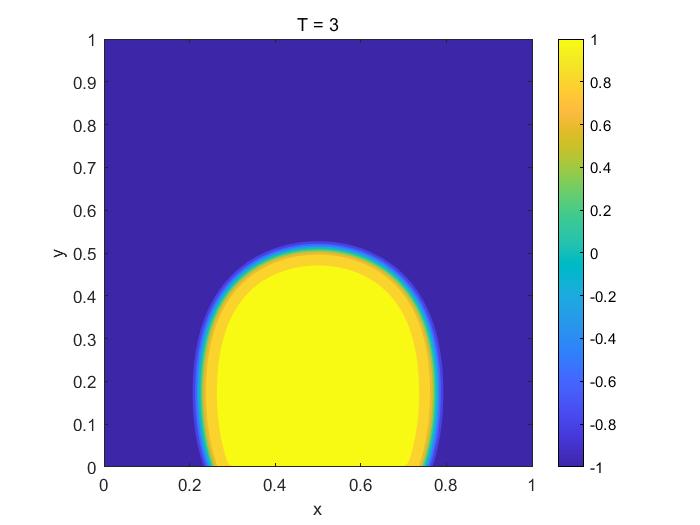}
    \includegraphics[width=0.23\linewidth]{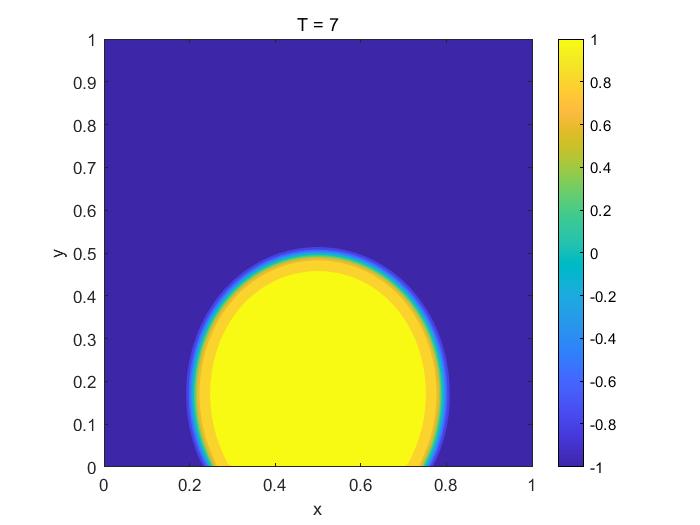}\\
    \hspace{-0.5cm} (A) $t=0.1$ \hspace{1.3cm} (B) $t=1$ \hspace{1.4cm} (C) $t=3$ \hspace{1.6cm} (D) $t=7$
    \caption{The contour plot of $\phi$ for the second example at times t=0.1, 1,  3, 7.}\label{ex2_phi_t}
\end{figure}

\begin{figure}[htbp]
\centering
\subfloat[mass evolution]{\label{fig:mass_ex2}\includegraphics[scale=0.11]{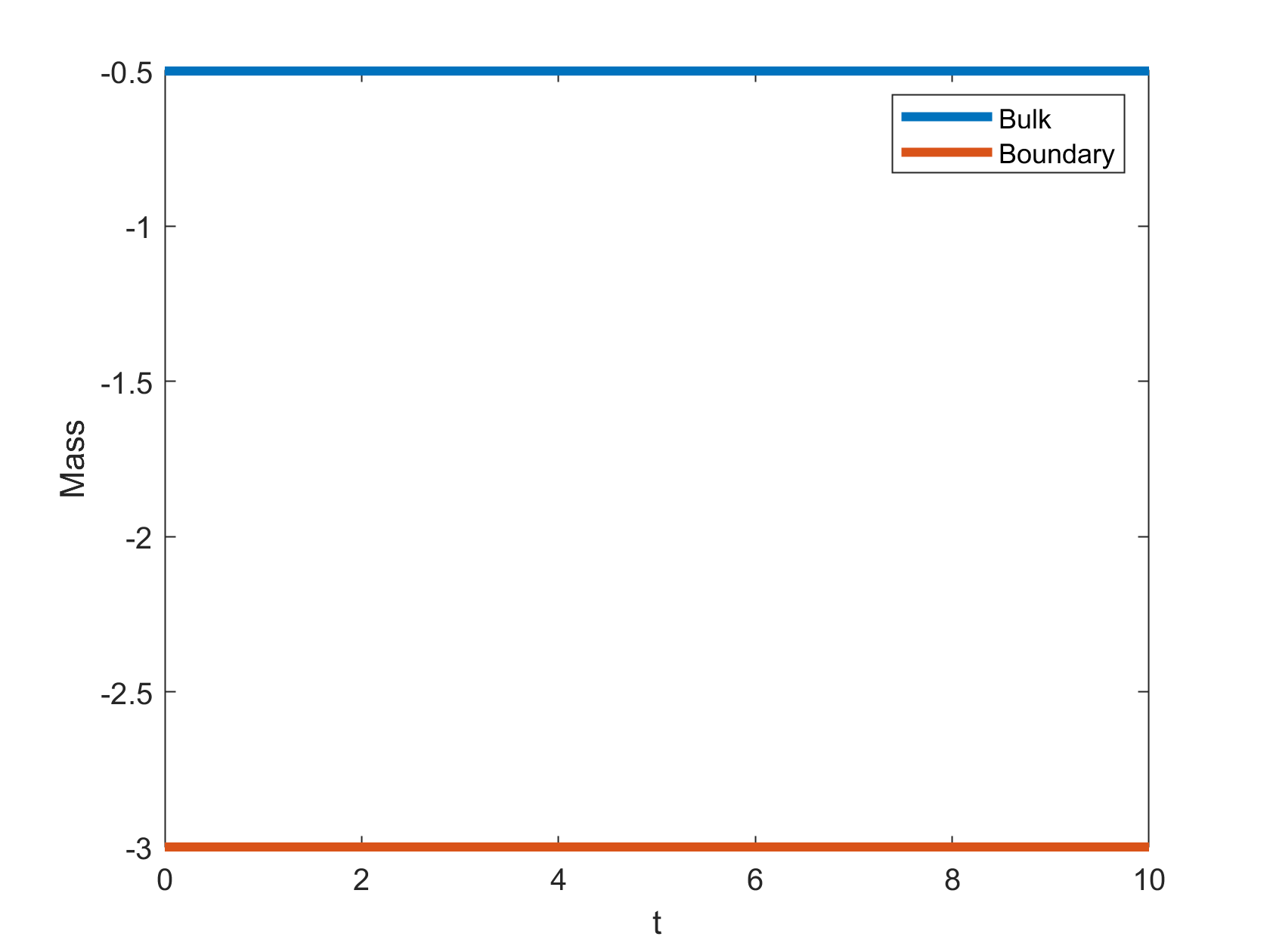}}
\subfloat[energy decay]{\label{fig:energy_ex2}\includegraphics[scale=0.2]{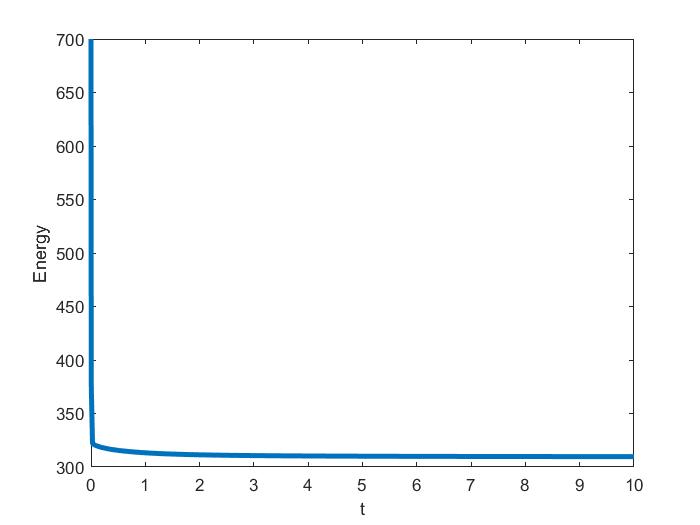}}
\caption{(A): Mass evolution with time in the bulk (blue line) and on the boundary (red line) for the second numerical example.
 (B): Energy decay over time.  }
 \label{num_ex2}
\end{figure}

\subsection{Example 5.3}
In the third example, the initial  state is:
\begin{equation}
    \phi_0(x,y) = 0.3+0.01*\cos{(6\pi x)}\cos{(6\pi y)},
\end{equation}
where $(x,y) \in \Omega = [0,1]\times [0,1].$
We show the evolution of $\phi$ at times $T=0.1, 0.5, 2, 7$ in Figure \ref{ex3_phi_t}. Initially, the $\phi$ field forms five distinct separation regions (see Figure \ref{ex3_phi_t} (A)), distributed periodically along the $x-$ and $y-$ directions within the domain. As time progresses, the central separation region vanishes, and the remaining four regions become more pronounced, aligning with the four boundaries of the domain (see Figure \ref{ex3_phi_t}(D)). Simultaneously, the area inside the boundaries gradually evolves into the 1-phase. 
Figure \ref{num_ex3} demonstrates the mass conservation and energy decay over time. 

 \begin{figure}[H]
    \centering
    \includegraphics[width=0.23\linewidth]{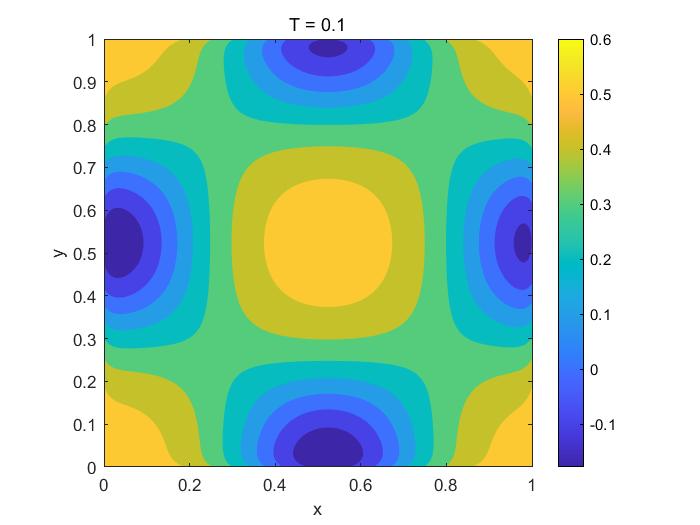}
    \includegraphics[width=0.23\linewidth]{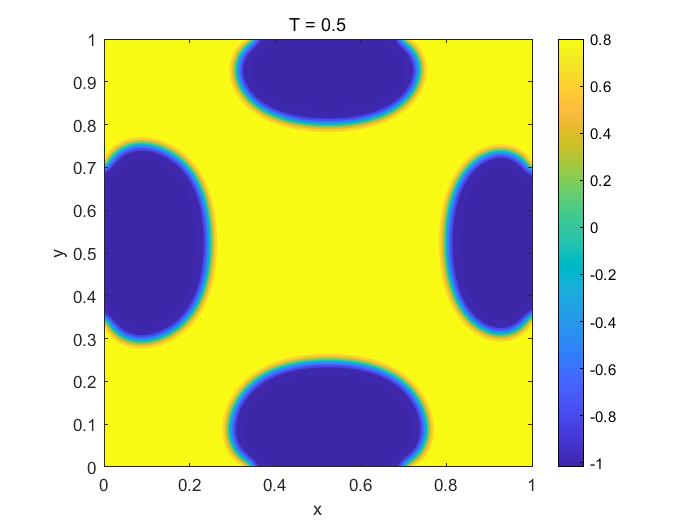}
    \includegraphics[width=0.23\linewidth]{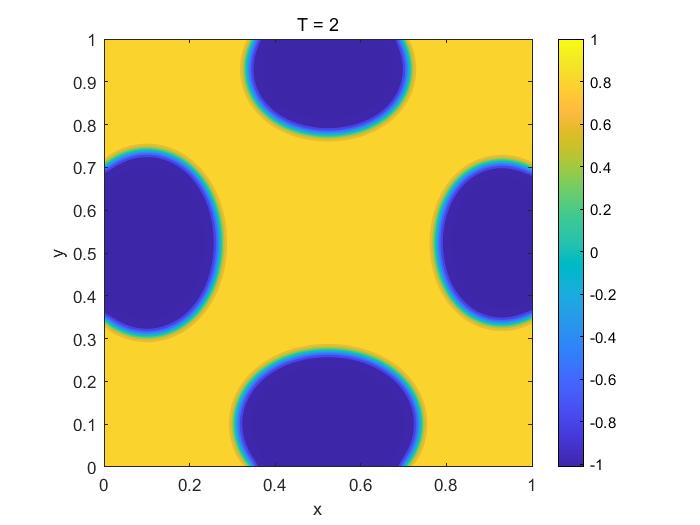}
    \includegraphics[width=0.23\linewidth]{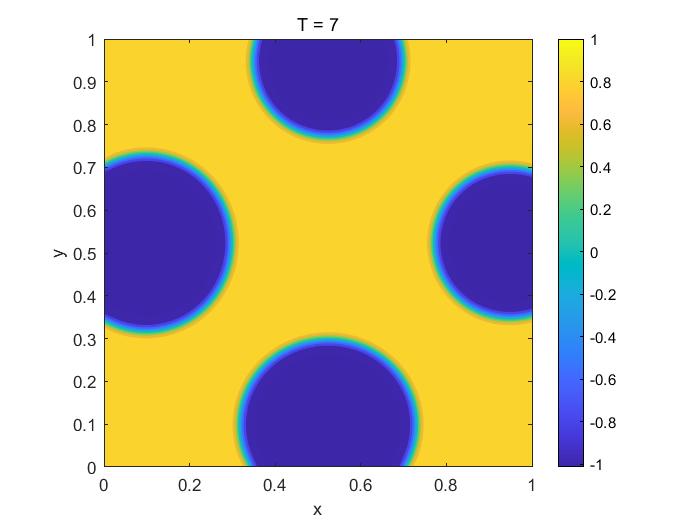}\\
    \hspace{-0.5cm} (A) $t=0.1$ \hspace{1.3cm} (B) $t=0.5$ \hspace{1.2cm} (C) $t=2$ \hspace{1.5cm} (D) $t=7$
    \caption{The contour plot of $\phi$ for the third example at times t=0.1, 0.5, 2, 7.}\label{ex3_phi_t}
\end{figure}

\begin{figure}[htbp]
\centering
\subfloat[mass evolution]{\label{fig:mass_ex3}\includegraphics[scale=0.11]{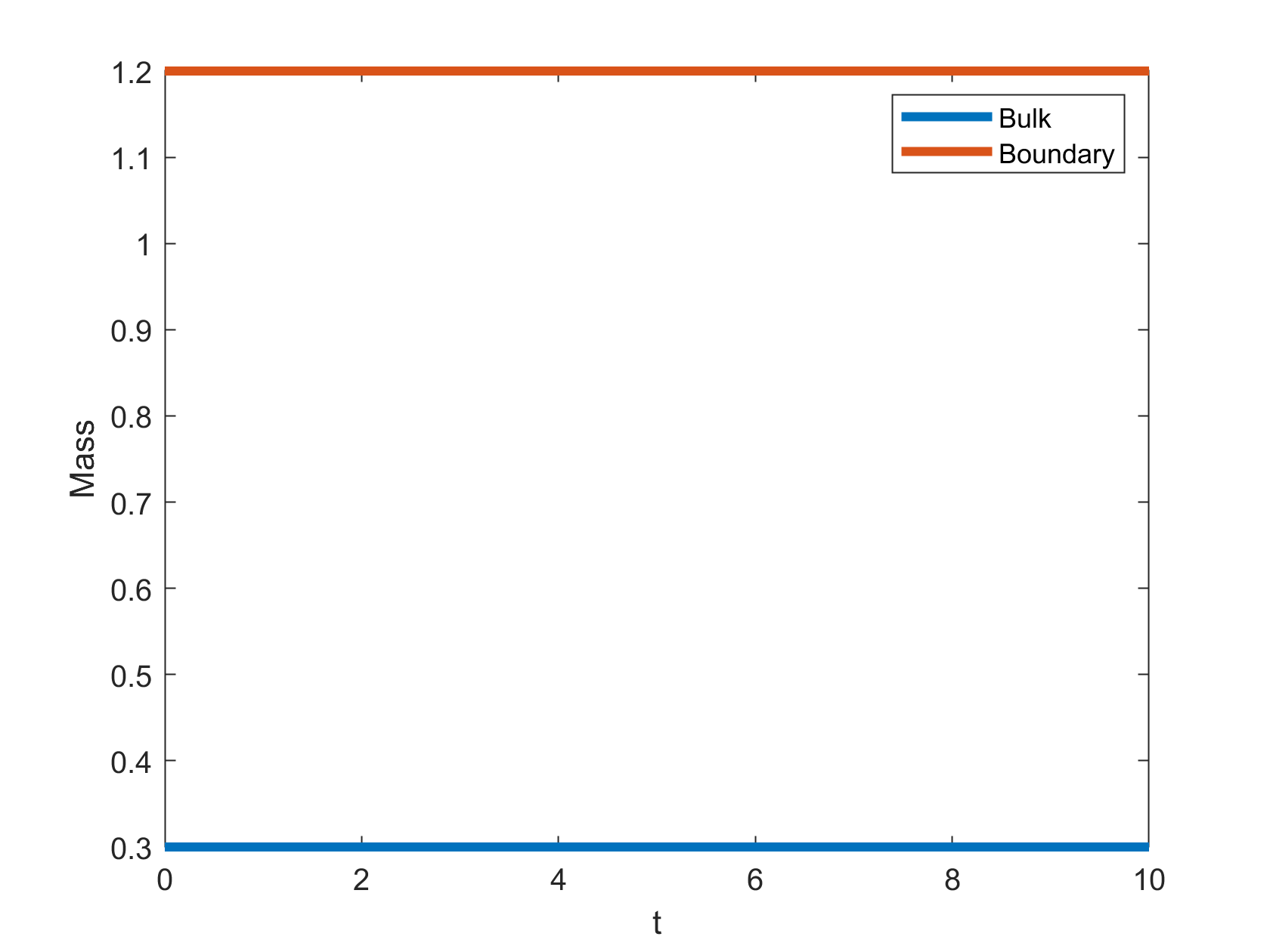}}
\subfloat[energy decay]{\label{fig:energy_ex3}\includegraphics[scale=0.2]{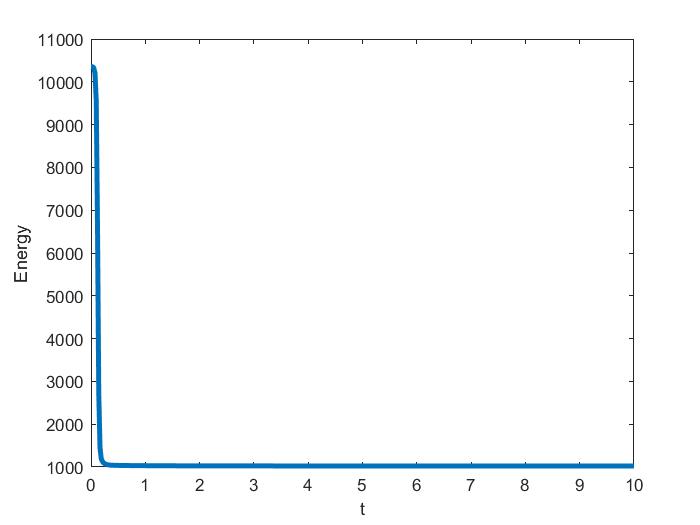}}
\caption{(A): Mass evolution with time in the bulk (blue line) and on the boundary (red line) for the third numerical example.
 (B): Energy decay with time.  }
 \label{num_ex3}
\end{figure}

\subsection{Example 5.4}
In the fourth example, the surface potential $G(\phi)$ is taken as the potential in moving contact line problems:
\begin{align*}
    G(\phi) = -\frac{\Tilde\gamma}{2} \cos(\theta_s) \sin(\frac{\pi}{2}\phi).
\end{align*}
Here, the domain size is $\Omega = [0,1]\times [0,1]$, and we consider two cases for $\theta_s$: $\theta_s = 30^{\circ}$ and $\theta_s = 150^{\circ}$. The initial conditions for both cases are the same, as shown in Figure \ref{phi0_mov}. In these two cases, we plot the evolution of $\phi$ at various times $T=1, 3, 5 $ and $10$, as shown in Figure \ref{fig:MCL1}. The top four states correspond to the case $\theta_s = 30^\circ$, and the bottom four states correspond to $\theta_s = 150^\circ$. It is clear that the steady-state angles differ between the two cases. Figure \ref{MCL30} and Figure \ref{MCL150} show the mass conservation and energy decay for the two cases $\theta_s = 30^{\circ}$ and $\theta_s = 150^{\circ}$, respectively.

\begin{figure}[htbp]
\centering
\includegraphics[scale=0.2]{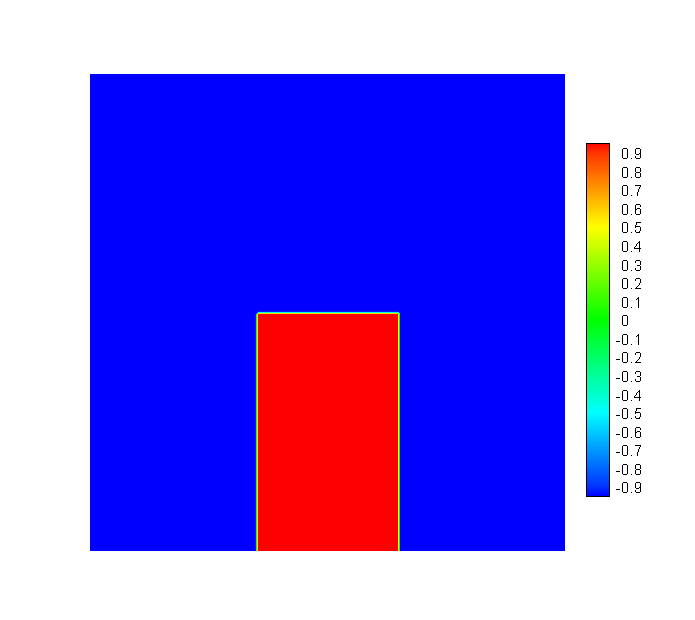}
\caption{Initial state for the moving contact line case.}\label{phi0_mov}
\end{figure}

\begin{figure}[htbp]
	\centering
	\includegraphics[width=0.2\linewidth, height=0.11\textheight]{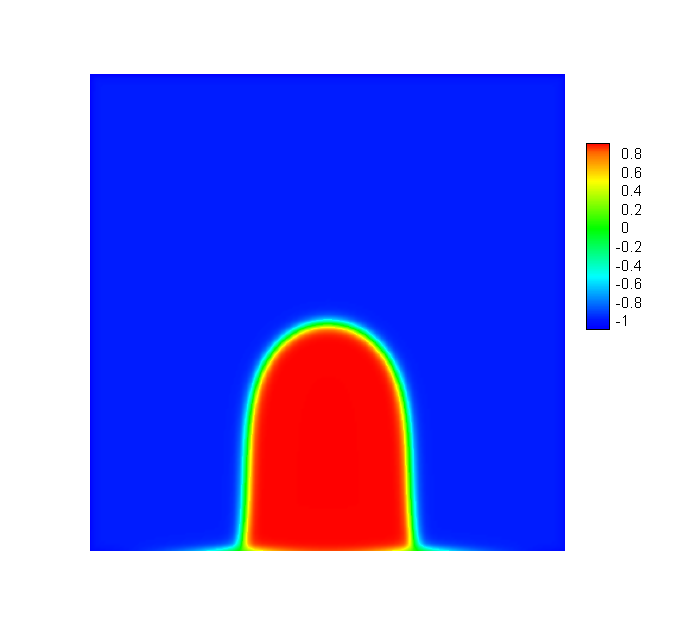}
	\includegraphics[width=0.2\linewidth, height=0.11\textheight]{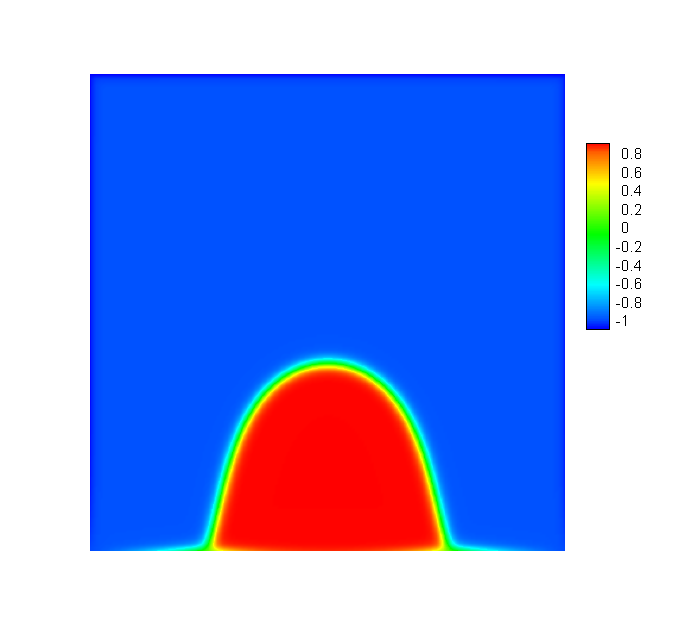}\includegraphics[width=0.2\linewidth, height=0.11\textheight]{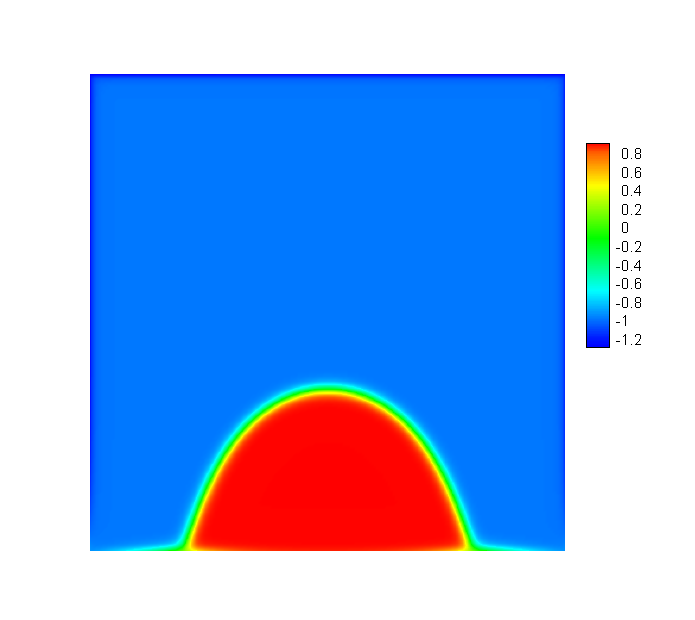} 
	\includegraphics[width=0.2\linewidth, height=0.11\textheight]{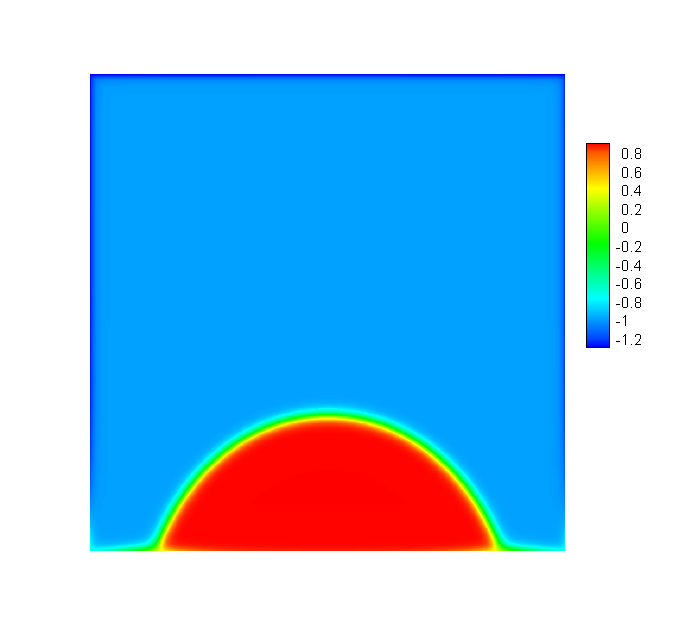} \\
	\includegraphics[width=0.2\linewidth, height=0.11\textheight]{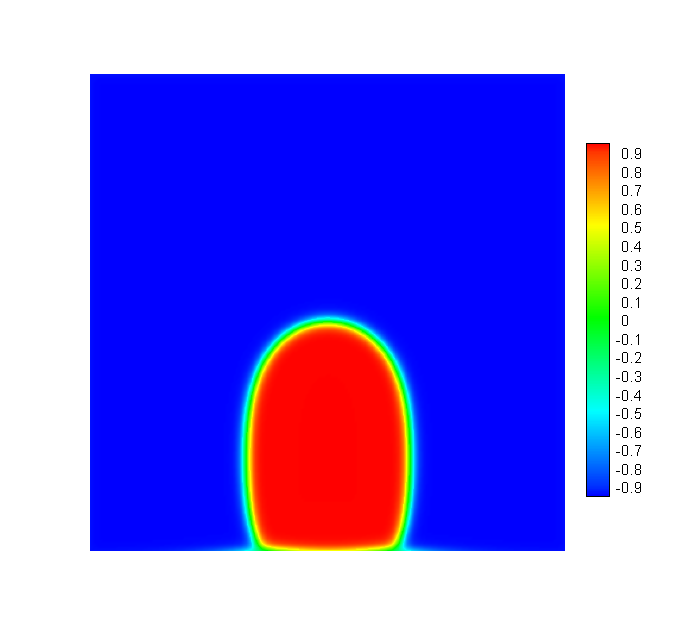} 
	\includegraphics[width=0.2\linewidth, height=0.11\textheight]{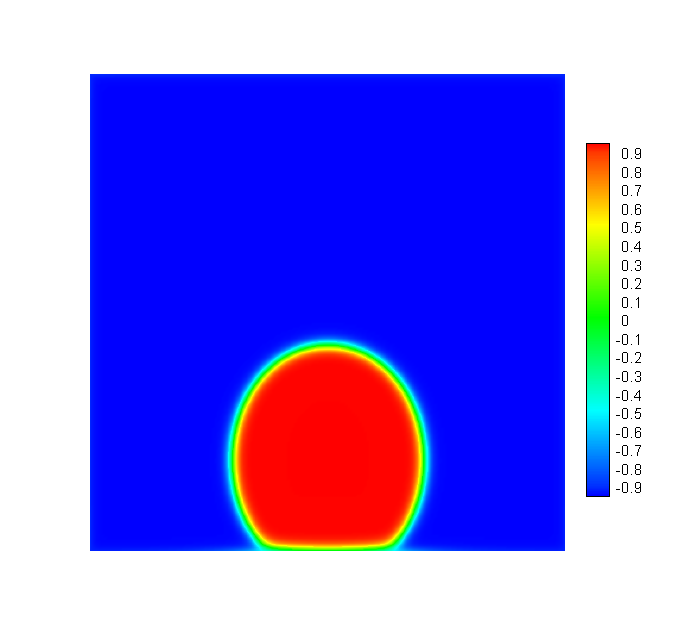}
    \includegraphics[width=0.2\linewidth, height=0.11\textheight]{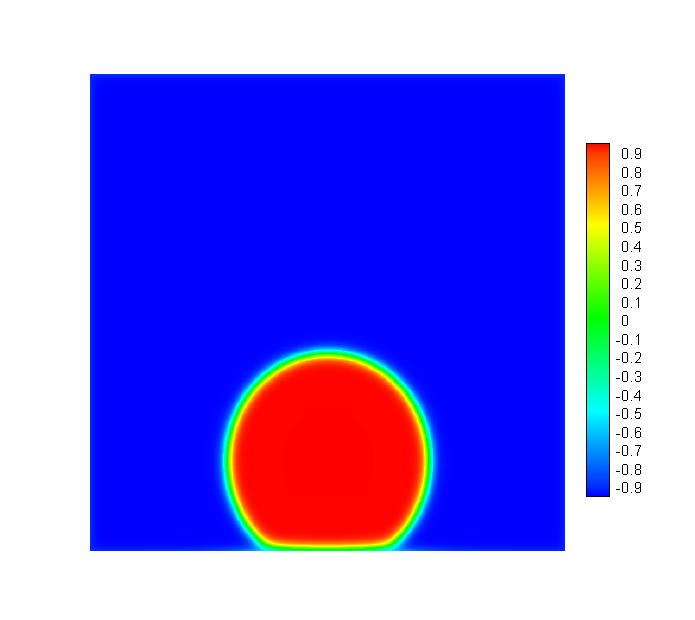} 
	\includegraphics[width=0.2\linewidth, height=0.11\textheight]{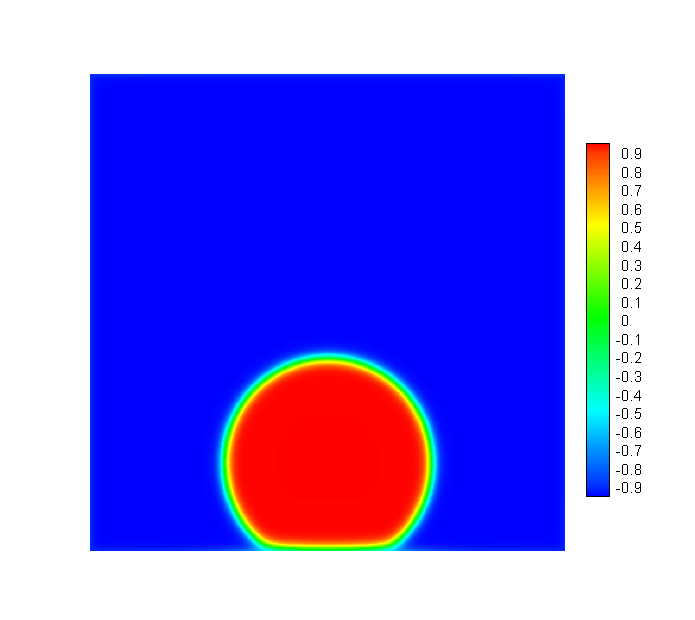} \\
    \hspace{-0.2cm} (A) $t=1$ \hspace{1.2cm} (B) $t=3$ \hspace{1.2cm} (C) $t=5$ \hspace{1.1cm} (D) $t=10$
	\caption{Snapshots of the numerical approximation at $t= 1$, $3$, $5$, and $10$. Top: $\theta_s=30^{\circ}$; Bottom: $\theta_s=150^{\circ}$.}
    \label{fig:MCL1}
\end{figure}

\begin{figure}[htbp]
\centering
\subfloat[mass evolution]{\label{fig:massMCL30a}\includegraphics[scale=0.11]{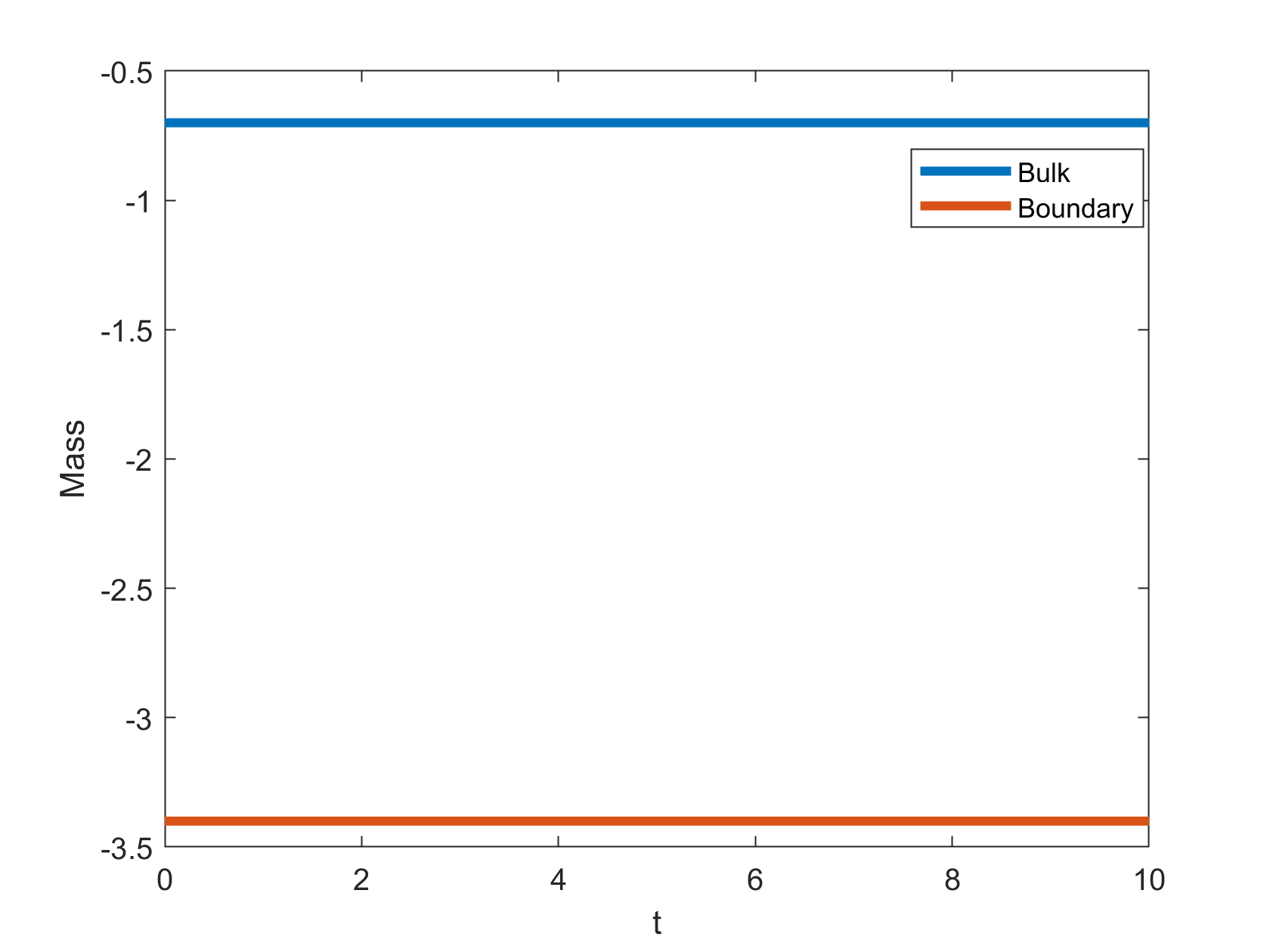}}
\subfloat[energy evolution]{\label{fig:massMCL30b}\includegraphics[scale=0.11]{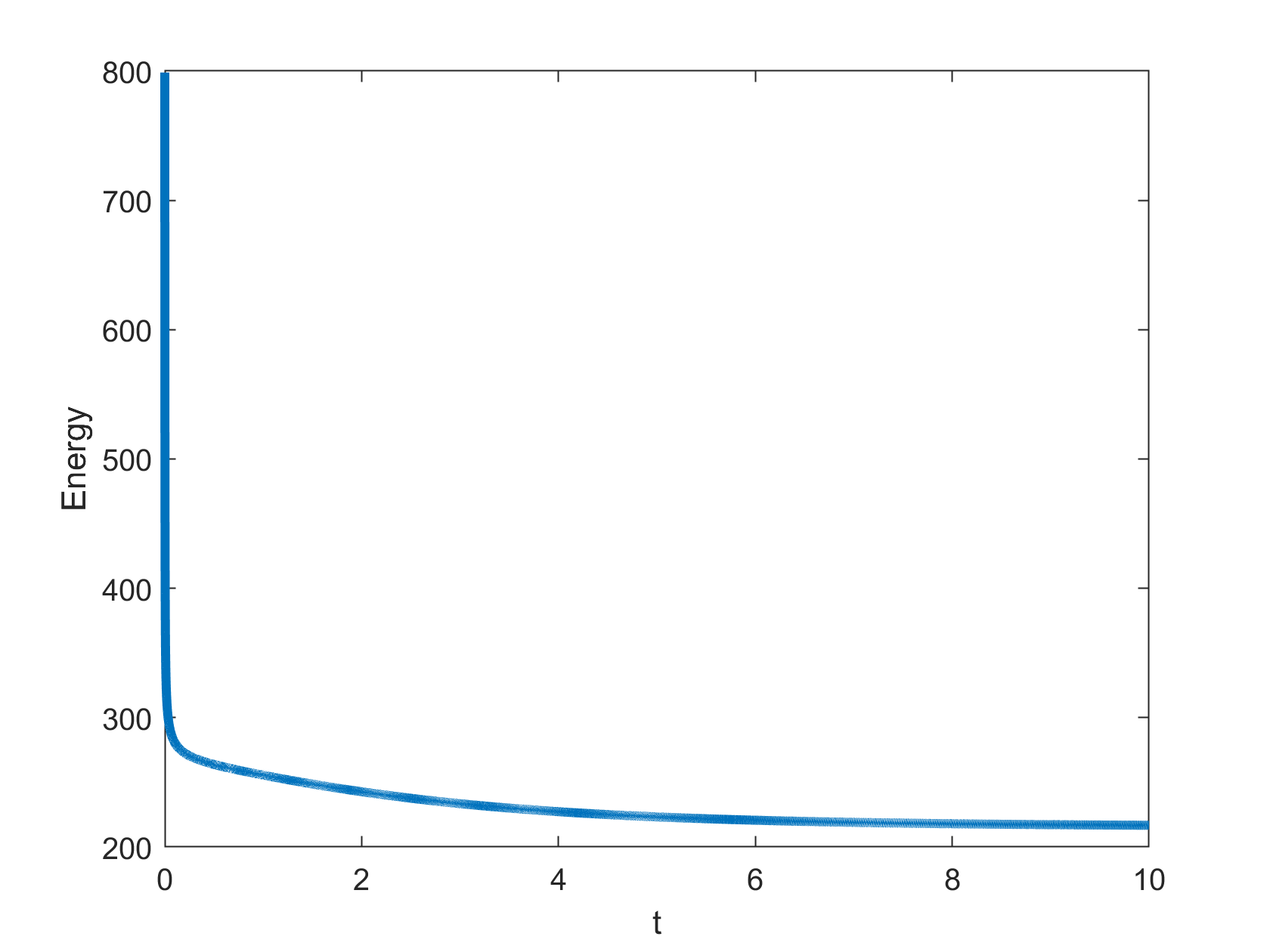}}
\caption{(A): Mass evolution over time in the bulk (blue line) and on the boundary (red line) for the case $\theta_s=30^{\circ}$.
 (B): Energy evolution over time for the case $\theta_s=30^{\circ}$.}
 \label{MCL30}
\end{figure}

\begin{figure}[htbp]
\centering
\subfloat[mass evolution]{\label{fig:massMCL150a}\includegraphics[scale=0.11]{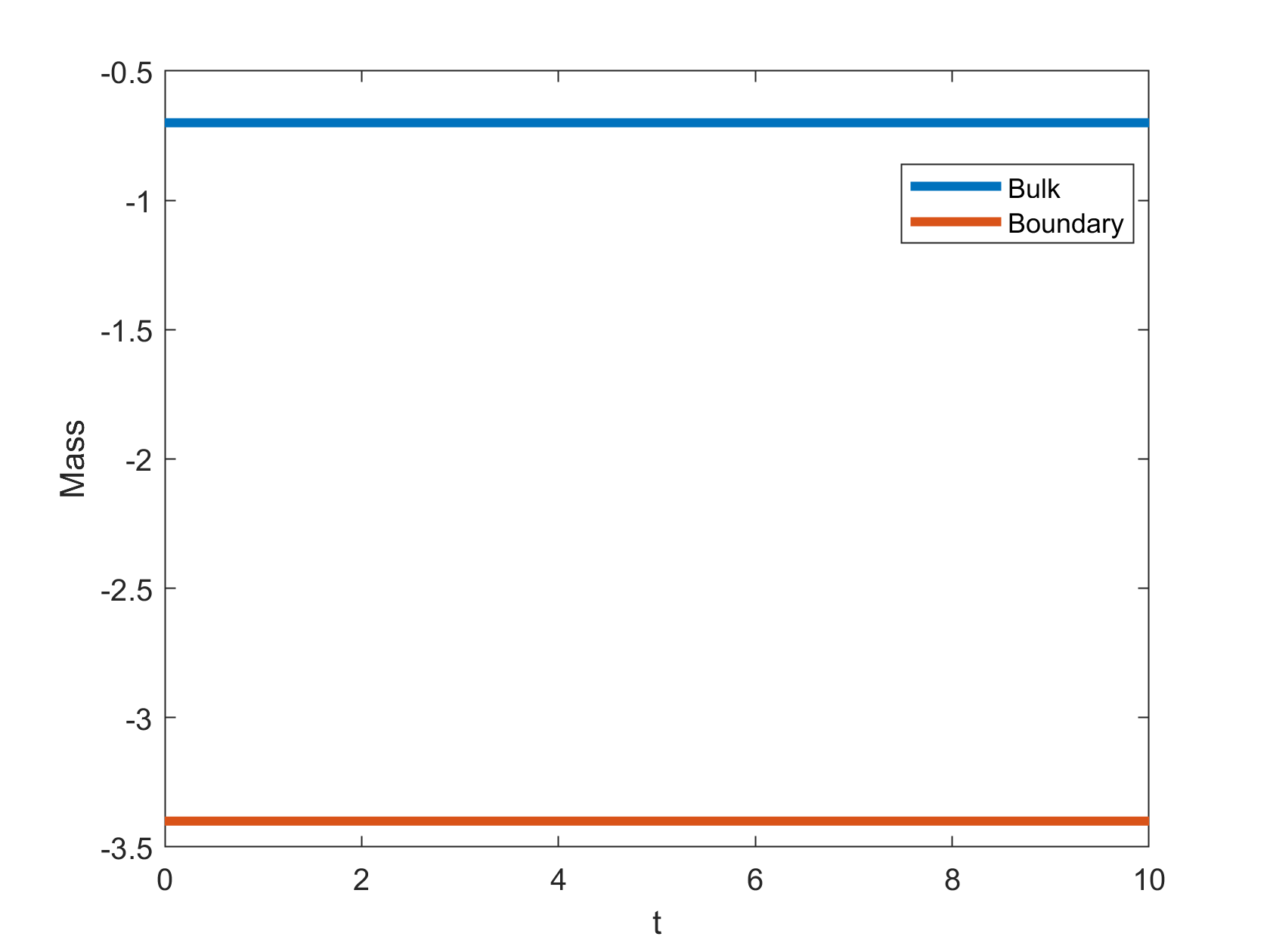}}
\subfloat[energy evolution]{\label{fig:massMCL150b}\includegraphics[scale=0.11]{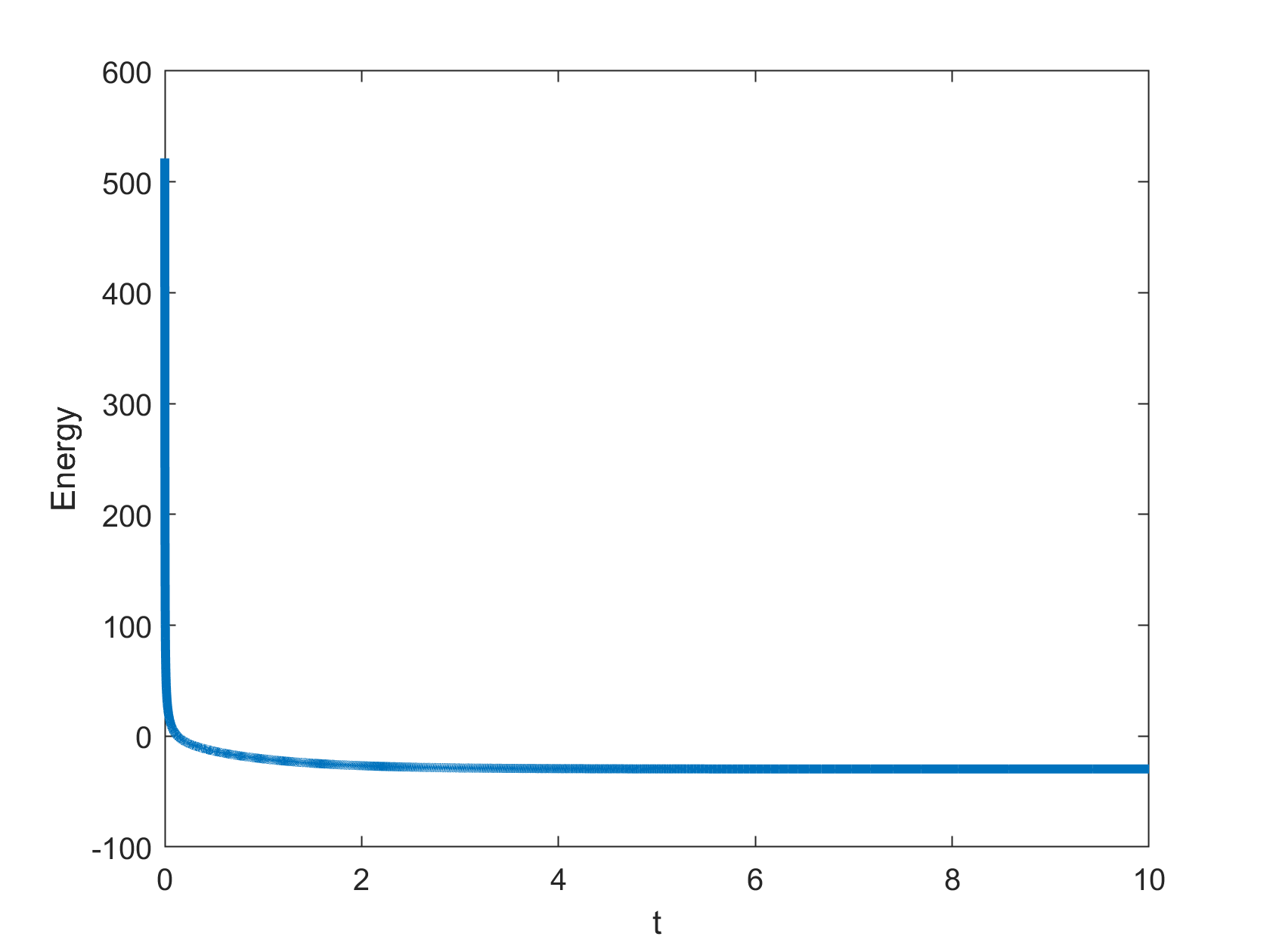}}
\caption{(A): Mass evolution over time in the bulk (blue line) and on the boundary (red line) for the case $\theta_s=150^{\circ}$.
 (B): Energy evolution over time for the case $\theta_s=150^{\circ}$.}
 \label{MCL150}
\end{figure}


\section{Conclusion}
\label{Conclusion}

In this work, we proposed a projected method for solving the Allen-Cahn (AC) equation with AC-type dynamic boundary conditions in order to calculate the steady states of the Cahn-Hilliard (CH) equation with CH-type dynamic boundary conditions. By introducing an orthogonal projection operator onto the confined subspace that satisfies mass conservation, the calculation of steady states in the $H^{-1}$ metric  is transformed equivalently to the calculation of steady states in the $L^2$ metric via projection.  This approach significantly reduces computational cost, as it
leads to a lower-order spatial derivative equation compared to directly working in the $H^{-1}$ metric.  More importantly, it avoids the need for computing the $\Delta^{-1}$ operator, which is  typically computationally expensive.
Additionally, we constructed an unconditionally energy-stable, time-discrete scheme using the convex splitting method.  This ensures stability and accuracy in the numerical solution. 
This projection idea has been applied successfully for the saddle point calculation of the CH equation with periodic boundary condition.

\section*{Acknowledgement}
Shuting Gu acknowledges the support of NSFC, China 11901211 and the Natural Science Foundation of Top Talent of SZTU, China GDRC202137. Rui Chen acknowledges the support of NSFC, China 12001055.

\newpage

\bibliography{CVXIMF,gad,ms,MsGAD,my, ProjAC_DBC}
 
\bibliographystyle{siamplain} 

\end{document}